\documentclass{amsart}
\usepackage{amssymb}
\usepackage{graphicx}
\usepackage[utf8]{inputenc}
\newtheorem{theorem}{Theorem}
\newtheorem{definition}{Definition}
\newtheorem{proposition}{Proposition}
\newtheorem{lemma}{Lemma}
\newtheorem{corollary}{Corollary}
\newtheorem{exam}{Example}
\newenvironment{example}{\begin{exam}\rm}{\end{exam}}
\newtheorem{exams}{Examples}

\newtheorem{rmk}{Remark}
\newenvironment{remark}{\begin{rmk}\rm}{\end{rmk}}
\newtheorem{notat}{Notation}

\renewcommand{\ne}{\not =}

\title[Invariant Hypersurfaces and Nodal Components of Foliations]{Invariant Hypersurfaces and Nodal Components for Codimension One Singular Foliations}
\author{F. Cano}
\author{J.F. Mattei}
\author{M. Ravara-Vago}
\date{\today}
\begin{document}
\begin{abstract}{It is known that there is at least an invariant analytic curve passing through each of the components in the complement of nodal singularities, after the reduction of singularities of a germ of singular foliation in ${\mathbb C}^2,0$}. Here, we state and prove a generalization of this property to any ambient dimension.
\end{abstract}
\maketitle
\tableofcontents
\section{Introduction}

This paper deals with the presence of invariant hypersurfaces in each component of the ``space of leaves'' of germs holomorphic codimension one foliations. More precisely, working in any ambient dimension,  we give a result that generalizes the two dimensional refined version of Camacho-Sad's theorem \cite{Cam-S}, stated an proved by Ortiz, Rosales and Voronin \cite{Ort-R-V}. The main result in this paper is the following one:

\begin{theorem}
\label{teo:main}
Consider a nodal reduction of singularities
$$
\pi:
(M, E,{\mathcal F})\rightarrow (({\mathbb C}^n,0),\emptyset,{\mathcal F}_0)
$$
of a
GH-foliation ${\mathcal F}_0$ on $({\mathbb C}^n,0)$. Let $\vert\mathcal S\vert $ be the support of the nodal separator set of $(M, E,{\mathcal F})$. For any connected component $C$ of $E\setminus \vert{\mathcal S}\vert $, there is an invariant hypersurface $H_0$ of ${\mathcal F}_0$ such that $H\cap C\ne\emptyset$, where $H\subset M$ is the strict transform of $H_0$ by $\pi$. Moreover, we have that $H\cap E\subset C$.
\end{theorem}
Concerning the existence of invariant hypersufaces, let us recall that it was a Thom's question if any  singular holomorphic foliation on $({\mathbb C}^2,0)$ has at least one invariant branch. A positive answer has been obtained by Camacho and Sad in \cite{Cam-S}. In higher ambient dimension, Jouanolou gave in \cite{Jou} examples of codimenson one dicritical holomorphic foliations without invariant hypersurface.  For any ambient dimension and non dicritical foliations, the existence of invariant hypersurface is proved in
\cite{Can-C, Can-M}.

 The space of leaves of foliations on $({\mathbb C}^2,0)$ is naturally separated by the so-called {``nodal points''}. This property has been remarked in \cite{Mat-M}.
After a ``nodal'' reduction of singularities, we can put all the nodal points  as ``corners'' of the exceptional divisor; when we remove them, the exceptional divisor is decomposed into several connected components: we get exactly ``$s+1$'' pieces if we have ``$s$'' nodal points.  By a result of Ortiz-Rosales-Voronin \cite{Ort-R-V}, each of these pieces intersects the strict transform of at least one invariant branch of the foliation. This statement can be considered as a ``refined version'' of Camacho-Sad's Theorem in \cite{Cam-S}. Let us remark that there are other two-dimensional versions \cite{Cam-R, Mat-M2}, where saddle nodes and dicritical components are also considered.

This paper provides an extension of the result in \cite{Ort-R-V} to any ambient dimension, for non-dicritical germs of codimension one singular foliations without saddle nodes. We refer to these conditions by saying that we have a ``non dicritical complex hyperbolic'' foliation, or a ``generalized hypersurface'', for short: a GH-foliation.

 We recall that the origin is a {\em nodal point} for  a foliation $\mathcal L$ on $({\mathbb C}^2,0)$ if, and only if, it is given in appropriate coordinates by $\omega=0$, where
$$
\omega=xdy-\lambda ydx
$$
and $\lambda$ is a positive irrational real number, see \cite{Mat-M}. In higher dimension, the objects that correspond to the nodal points, from the view point of the separating properties in the space of leaves, are  called {\em nodal separating blocks}.
Let us give a quick description of what they are.

Consider a foliated space $(M,E,{\mathcal F})$, given by a nonsingular complex analytic space $M$, a normal crossings divisor $E\subset M$ and a codimension one foliation $\mathcal F$ on $M$. Assume that it is a desingularized  $GH$-foliated space. This means that the irreducible components of $E$ are invariant, there are no saddle nodes and
all the points of $M$ are simple for $({\mathcal F}, E)$ in the sense of \cite{Can, Can-C}.
Any irreducible component $\Gamma$ of the singular locus $\mbox{\rm Sing}({\mathcal F})$ is a codimension two non singular subspace $\Gamma\subset M$.
We say that $\Gamma$ is of {\em nodal type}, respectively {\em real saddle type}, if the generic transversal type of $\mathcal F$ with respect to $\Gamma$ is a two dimensional nodal foliation, respectively a two dimensional real saddle  foliation (negative real quotient of eigenvalues).
Let ${\mathcal N}$ be the union of the nodal type irreducible components $\Gamma$ of $\mbox{\rm Sing}({\mathcal F})$.
A connected component $\mathcal B$ of ${\mathcal N}$ is a {\em nodal separating block} when it only intersects nodal and real saddle type irreducible components of the singular locus ( in
\cite{Can-RV,Can-RV-S} we call these sets ``uninterrupted nodal components''). The {\em nodal separator set } $\mathcal S$ is the union of all nodal separating blocks.

Nodal separating blocks separate locally the space of leaves exactly in the same way as a nodal type point in dimension two. Nevertheless, it should be possible that a contracting  ``secondary'' or ``singular'' holonomy  allows the passage of leaves through a given separating block.  This is a behavior of global nature. Although we can produce examples ``ad hoc'' of this kind of singular holonomy, we have no examples in the case of spaces obtained by reduction of singularities of a foliation on $({\mathbb C}^n,0)$.

We say that a desingularized foliated space is {\em nodally desingularized} when the nodal separating blocks are of {\em corner type}. This means that each irreducible component $\Gamma$ of the separating blocks is contained in two irreducible components of the divisor.
A {\em nodal reduction of singularities of a foliation ${\mathcal F}_0$ on $({\mathbb C}^n,0)$} is a morphism
$$
\pi:(M,E, {\mathcal F})\rightarrow (({\mathbb C}^n,0),\emptyset,{\mathcal F}_0)
$$
that is the composition of a finite sequence of ``admissible'' blow-ups and such that $(M,E,{\mathcal F})$ is nodally desingularized. In Subsection
\ref{Nodally Desingularized Foliated Spaces},
we show the existence of nodal reduction of singularities for GH-foliations on $({\mathbb C}^n,0)$, when we dispose of a reduction of singularities of the set of invariant hypersurfaces.

It is now natural to ask how many connected components of $E\setminus {\mathcal S}$ we find, when we consider a non-dicritical CH-foliated space that is nodally desingularized. This gives an accurate sense to Theorem \ref{teo:main}.
The topological combinatorics of the exceptional divisor allows us to prove that each nodal separating block divides the exceptional divisor into two pieces, exactly as in the case of dimension two. More precisely, we have the following result:
\begin{theorem}
\label{teo:numerodecomponents} Let
$
(M,E,{\mathcal F})\rightarrow (({\mathbb C}^n,0),\emptyset,{\mathcal F}_0)
$
be a
nodal reduction of singularities
of a GH-foliation ${\mathcal F}_0$ of $({\mathbb C}^n,0)$.  Let $\mathcal S$ be the separator set  of $(M,E,\widetilde{\mathcal F})$. If $\mathcal S$
has ``$s$'' nodal separating blocks, then $E\setminus {\mathcal S}$ has exactly ``$s+1$'' connected components.
\end{theorem}
The proof of Theorem \ref{teo:numerodecomponents} is given by arguments of combinatorial topology, by showing that we are in a situation very similar to the classical Jordan's Curve Theorem. To do this, we work in a systematic way with the stratification of $M$ induced by the exceptional divisor $E$.

The proof of Theorem \ref{teo:main} goes by an inductive argument on the length of a reduction of singularities and also by considering the two dimensional case based in an adecuate notion of $2$-equireduction.

{\em Acknowledgements:} We thank Etienne Fieux for many helpful discussions and the algebraic topology clarifications he has brought to us, they have been very valuable for the proof of Theorem \ref{teo:numerodecomponents}.

\section{Preliminaries}
\label{Preliminaries}
 We recall that a codimension one singular foliation $\mathcal F$ of $({\mathbb C}^n,0)$ is generated by a germ of nonzero differential $1$-form
$$
\omega=a_1dx_1+a_2dx_2+\cdots +a_ndx_n;\quad a_i\in {\mathcal O}_{{\mathbb C}^n,0},\; i=1,2,\ldots,n,
$$
satisfying Frobenius' integrability condition $\omega\wedge d\omega=0$ and such that the coefficients $a_i$ have no common divisor. We also denote $\mathcal F$ by $\omega=0$. The singular locus $\mbox{\rm Sing}\mathcal F$ is the subspace of $({\mathbb C}^n,0)$ defined by $$a_1=a_2=\cdots=a_n=0.$$
It is an analytic germ in $({\mathbb C}^n,0)$ of codimension greater or equal than two.

 An analytic morphism $\phi:({\mathbb C}^m,0)\rightarrow ({\mathbb C}^n,0)$ is {\em invariant} by $\mathcal F$ when $\phi^*\omega=0$. We say that an analytic germ $(Z,0)\subset ({\mathbb C}^n,0)$ is {\em invariant} by $\mathcal F$ if and only if any analytic morphism factorizing through $(Z,0)$ is invariant. In particular, we see that a germ of hypersurface
$(H,0)\subset ({\mathbb C}^n,0)$, given by the reduced equation $f=0$, is {\em invariant by $\mathcal F$} if and only if  $f$ divides the differential $3$-form $\omega\wedge df$.

For any non invariant morphism $\phi:({\mathbb C}^m,0)\rightarrow ({\mathbb C}^n,0)$, the
{\em pull-back $\phi^*{\mathcal F}$} is a well defined codimension one singular foliation of $({\mathbb C}^m,0)$. We obtain a generator of $\phi^*{\mathcal F}$ after dividing $\phi^*\omega$ by the maximum common divisor of its coefficients.

Let us recall the definition of {\em dicritical singular foliation}, see \cite{Can-C}:

\begin{definition} A codimension one singular foliation $\mathcal F$ on $({\mathbb C}^n,0)$ is {\em dicritical} if, and only if, there is a non invariant analytic morphism
$
\phi:({\mathbb C}^2,0)\rightarrow ({\mathbb C}^n,0)
$ and suitable coordinates $x,y$ in $({\mathbb C}^2,0)$,
such that:
 \begin{enumerate}
 \item $\phi^*{\mathcal F}$ is the foliation given by $dy=0$.
 \item The morphism $(\phi\circ\gamma):({\mathbb C},0)\rightarrow ({\mathbb C}^n,0)$ is invariant, where
$\gamma(t)=(0,t)$. {\em (Let us note that $\gamma$ is not invariant for $\phi^*{\mathcal F}$).}
 \end{enumerate}
\end{definition}

In this paper, we are interested in foliations ``without saddle nodes''. Let us precise this. We recall that a {\em saddle node} is a singular foliation of $({\mathbb C}^2,0)$, defined by a differential $1$-form $\omega$, such that
$$
\omega=xdy+(\mbox{ terms of order }\geq 2).
$$
In dimension two, the singular foliations without saddle nodes in their reduction of singularities are usually called {\em generalized curves}. They can be dicritical or not. In the nondicritical case, these foliations have the same reduction of singularities as the set of invariant curves, as shown in \cite{Cam-L-S}. In higher dimension, we consider the following definition
\begin{definition} Let $\mathcal F$ be a codimension one singular foliation on $({\mathbb C}^n,0)$. We say that $\mathcal F$ is {\em complex hyperbolic}, for short a {\em CH-foliation}, if, and only if, there is no analytic morphism $\phi:({\mathbb C}^2,0)\rightarrow ({\mathbb C}^n,0)$ such that $\phi^*{\mathcal F}$ is a saddle node.
\end{definition}
\begin{remark}
\label{rk:ch}
Any reduction of singularities of the set of invariant hypersurfaces of a nondicritical CH-foliation provides a reduction of singularities of the foliation, see for instance \cite{Fer-M}. In this sense, they can be called {\em generalized hypersurfaces}. This is no longer true in the dicritical case, since there are examples without  invariant hupersurfaces \cite{Jou}.
\end{remark}

\section{Standard Ambient Spaces}
We consider ambient spaces obtained from $({\mathbb C}^n,0)$ by finite sequences of blow-ups, with centers having normal crossings properties. Let us precise the definitions.

Consider a finite family ${\mathcal Y}=\{Y_j\}_{j\in J}$ of closed analytic subsets $Y_j$ of a given nonsingular complex analytic space $M$. We say that ${\mathcal Y}$ has {\em local normal crossings at $p\in M$} if the following property holds:
\begin{quote}
There is a local coordinate system $\boldsymbol{z}=(z_1,z_2,\ldots,z_n)$ of $M$ at $p$, satisfying that each $Y_j$ is given by a local equation $
(z_i=0;\; i\in A_j)
$, where $A_j\subset\{1,2,\ldots, n\}$, for any $j\in J$.
\end{quote}
 The family ${\mathcal Y}$ has {\em locally normal crossings in $M$} if it has local normal crossings at any $p\in M$. Let us note that in this case the closed analytic sets
 $Y_K=\cap _{j\in K}Y_j$ are non singular. We say that ${\mathcal Y}$ has {\em normal crossings in $M$} if, in addition, each $Y_K$ is connected for any $K\subset J$.

A {\em standard ambient space} ${\mathcal M}=(M,E;K)$ is the data of a nonsingular germ of complex analytic space $(M,K)$ and divisor $(E,E\cap K)\subset (M,K)$, such that $(M,K)$ is isomorphic to $({\mathbb C}^n,0)$ when $E=\emptyset$ and if $E\ne\emptyset$, we have:
\begin{enumerate}
\item The divisor $E$ is a union $E=\cup_{i\in I}E_i$ of non singular connected hypersurfaces $E_i$. The {\em soul $K$} is a connected union $K=\cup_{i\in I}K_i$ of connected and compact non singular analytic spaces $K_i$, with $K_i=E_i\cap K$.
\item The family $\{E_i, K_j\}_{i,j\in I}$ has normal crossings in $M$.
\end{enumerate}
In particular, the components $E_i$ of $E$ are of two types. If $K_i=E_i$ and hence $E_i\subset K$, we have a  compact component $E_i$. When $K_i\ne E_i$, we have a non compact germ $(E_i,K_i)$ over the compact analytic subset $K_i$.

Consider a {standard ambient space} ${\mathcal M}=(M,E;K)$. We recall that a {\em closed immersion} of germs $(Y,T)\subset (M,K)$ is obtained by a closed immersion between suitable ``small enough'' representatives $Y\subset M$ in such a way that $T=Y\cap K$.
We say that $(Y,T)$ has {\em normal crossings with ${\mathcal M}$} if the family
$$
\left\{
Y, T\}\cup\{E_i, K_j
\right\}_{i,j\in I}
$$
has  normal crossings in $M$.
In this situation, the blow-up
$
(M',K')\rightarrow (M,K)
$
of $(M,K)$ with center $(Y,T)$, induces a transformation of standard ambient spaces
$$
\pi: (M',E';K')\rightarrow  (M,E;K);\quad E'=\pi^{-1}(E\cup Y),\; K'= \pi^{-1}(K).
$$
Such transformations are called {\em standard blow-ups} and the center $(Y,T)$ is called a {\em standard center for $\mathcal M$}. Any composition of a finite sequence of standard blow-ups is called {\em a standard transformation}.

Consider a standard ambient space ${\mathcal M}=(M,E;K)$ and let ${\mathcal L}=\{H_1,H_2,\ldots,H_k\}$ be a finite list of closed irreducible hypersurfaces $(H_i,H_i\cap K)\subset (M,K)$ and let us put $H=H_1\cup H_2\cup \cdots \cup H_k$.
We say that the pair $({\mathcal M},{\mathcal L})$ is {\em simple} if the following properties hold:
\begin{enumerate}
\item $H_i \cap H_j= \emptyset$, for $i\ne j$.
\item $H_i\not\subset E$, for $i=1,2,\ldots,k$.
\item The space ${\mathcal M}_{\mathcal L}=(M,E\cup H;K)$ is a standard ambient space.
If we only ask $E\cup H$ to have local normal crossings, we say that $({\mathcal M},{\mathcal L})$ is {\em locally simple}.
\end{enumerate}
A standard center $(Y,T)$ for $\mathcal M$ is {\em admissible} for the pair $({\mathcal M},{\mathcal L})$ if $Y$ is contained in the singular locus of $E\cup H$. As before, an {\em admissible blow-up for $({\mathcal M},{\mathcal L})$}
 is a blow-up with an admissible center
 $
 ({\mathcal M}',{\mathcal L}')\rightarrow ({\mathcal M},{\mathcal L})
 $,
  where ${\mathcal L}'$ is the list of the strict transforms of the elements in $\mathcal L$. An admissible transformation is the composition of a finite sequence of admissible blow-ups.

 \begin{definition}
   Let $({\mathcal M}_0,{\mathcal L}_0)$ be a pair where $\mathcal M=(M_0,E^0)$ is an ambient space and ${\mathcal L}_0$ is a finite list of irreducible hypersurfaces of
   ${\mathcal M}_0$ not contained in $E^0$. An {\em admissible reduction of singularities} of $({\mathcal M}_0,{\mathcal L}_0)$ is any admissible transformation $({\mathcal M},\mathcal L)\rightarrow ({\mathcal M}_0,{\mathcal L}_0)$ such that $({\mathcal M},\mathcal L)$ is simple.
 \end{definition}
 \begin{remark}
 \label{rk:reductionoflists}
 The existence of an admissible reduction of singularities for a given pair $({\mathcal M}_0,{\mathcal L}_0)$ is a problem hugely close to the classical Hironaka's reduction of singularities in \cite{Hir} and \cite{Aro-H-V}. In dimension three, for $M_0=({\mathbb C}^3,0)$, it is possible to show directly the existence of an admissible reduction of singularities, we give an outline of a proof in Appendix
 \ref{Appendix: Strong Desingularization in Dimension Three}.
 In dimension $n\geq 4$, there are no explicit statements in the literature and a proof is maybe possible by introducing the global condition of connectedness in the known procedures. Note also that in view of Remark \ref{rk:tresdimensionalcase}, most of the technical difficulties are concentrated in dimension three. Anyway, along this paper we work under the assumption that the following statement is true:
 \begin{quote}\em
 Consider an ambient space ${\mathcal M}_0$ and a finite list
 ${\mathcal L}_0$ of irreducible hypersurfaces of ${\mathcal M}_0$ not contained in $E^0$. There is an admissible reduction of singularities of $({\mathcal M}_0,{\mathcal L}_0)$.
 \end{quote}
 \end{remark}

\section{Desingularized  GH-Foliated Spaces}
Let us recall from the Introduction that a {\em foliated space} is a pair $({\mathcal M},\mathcal{F})$ where ${\mathcal M}=(M,E;K)$ is a standard ambient space and $\mathcal F$ is a codimension one singular holomorphic foliation on $(M,K)$.
In this paper, we consider only {\em non-dicritical CH-foliated spaces}. This means that ${\mathcal F}$ is a non-dicritical CH-foliation and moreover each irreducible component of $E$ is invariant for $\mathcal F$.  From now on, we refer to the non-dicritical CH-foliated spaces as {\em generalized hypersurface type foliated spaces}, for short: {\em  GH-foliated spaces}.

It is important to note that the class of GH-foliated spaces is stable under standard blow-ups with invariant center.

\subsection{Simple Points}
We particularize here the definition of simple point in \cite{Can-C,Can} to the case of a GH-foliated space $({\mathcal M},{\mathcal F})$.

Let us consider a point $p\in K\subset M$. The {\em dimensional type $\tau_p{\mathcal F}$} is the dimension of the ${\mathbb C}$-vector space $T_p{\mathcal F}$ given by the vectors $\xi(p)\in T_p{\mathbb C}^n$, where $\xi$ is a germ of vector field tangent to $\mathcal F$.
Denote $\tau=\tau_p{\mathcal F}$ and let $e=e_p(E)$ be the number of irreducible components of $E$ through $p$. There is a local coordinate system $(x_i)_{i=1}^n$  and an integrable germ of $1$-form $\omega$ defining $\mathcal F$, such that $E=(\prod_{i=1}^ex_i=0)$ and
$$
\omega=\sum_{i=1}^\tau a_idx_i,\quad \partial a_i/\partial x_j=0,\, j=\tau+1,\tau+2,\ldots,n,
$$
where the coefficients $a_i$ are without common factor. Since the irreducible components of $E$ are invariant, we can write $\omega$ in a logarithmic way as $\omega=(
{\prod}_{i=1}^e x_i)\eta$, where
$
 \eta=
\sum_{i=1}^{e}b_i({dx_i}/{z_i})+
\sum_{i=e+1}^\tau b_idx_i
$ and the coefficients $b_i$ are without common factor. We say that $p$ is a {\em corner point } for $({\mathcal M},{\mathcal F})$ when $e=\tau$. In this case, we have that
$
 \eta=
\sum_{i=1}^{\tau}b_i{dx_i}/{z_i}
$
and there is a {\em residual vector}  $\lambda$ given by  $\lambda=(b_1(0),b_2(0),\ldots,b_\tau (0))$.
\begin{proposition}
 \label{pro:simplecorner}
 Let $p\in K\subset M$ be a corner point for a GH-foliated space $({\mathcal M},{\mathcal F})$. The following statements are equivalent:
\begin{enumerate}
\item The residual vector $\lambda=(\lambda_1,\lambda_2,\ldots,\lambda_\tau)$ is non null, that is $\lambda\ne\mathbf{0}$.
\item We have $
    \sum_{i=1}^\tau m_i\lambda_i\ne0
    $,  for any $m\in {\mathbb Z}^{\tau}_{\geq 0}\setminus \{\mathbf{0}\}$.
\item The only germs of invariant hypersurfaces of $\mathcal F$ through $p$ are the irreducible components of the exceptional divisor.
\end{enumerate}
\end{proposition}
\begin{proof}See \cite{Fer-M}.
\end{proof}
\begin{definition}
 \label{def:simplecorner}
 Let $({\mathcal M},{\mathcal F})$ be a GH-foliated space and consider a point $p\in K\subset M$. We say that $p$ is a {\em simple corner} for $({\mathcal M},{\mathcal F})$  if and only if it is a corner point and the equivalent properties in the statement of Proposition  \ref{pro:simplecorner} are satisfied. We say that $p$ is a {\em simple trace point} for $({\mathcal M},{\mathcal F})$ if there is a nonsingular germ of invariant hypersurface $(H,p)$, not contained in $E$, such that $(E\cup H,p)$ is a normal crossings divisor of $(M,p)$ and $p$ is a simple corner for the new GH-foliated space
 $
 ((M,E\cup H;\{p\}),{\mathcal F})
 $. We say that $p$ is a simple point iff it is either a simple corner or a simple trace point.
\end{definition}

Note that the residual vector $\lambda$ is well defined as well for the case of simple trace points. Take a simple point $p$ for $({\mathcal M},{\mathcal F})$. We say that $p$ is of {\em real type} if the quotients $\lambda_i/\lambda_j$ are real numbers. In this case, we have two possibilities:
\begin{enumerate}
\item The point is of  {\em real saddle type} if $\lambda_i/\lambda_j>0$ for any $i,j$.
\item The point is of {\em nodal type} if there is at least one negative quotient $\lambda_i/\lambda_j<0$.
\end{enumerate}

\begin{remark} Assume that $p\in K$ is a point of nodal type and that $\omega$ is a local generator of $\mathcal F$ at $p$.
It is known that $\omega$ can be ``linearized''\cite{Cer-L, Cer-M}. That is, we can write $\omega=({\prod}_{i=1}^\tau x_i)\eta$ with
$\eta=\sum_{i=1}^\tau\lambda_i{dx_i}/{x_i}$,
where $\lambda_j/\lambda_i\in {\mathbb R}^*$ and there is at least one pair $i,j$ such that $\lambda_j/\lambda_i<0$. In this case, the invariant real hypersurface
$
{\prod}_{i=1}^\tau\vert x_i\vert^{\lambda_j/\lambda_1}=1
$
defines locally a ``separation'' in the space of leaves. Moreover, the components of $\operatorname{Sing}({\mathcal F})$ through the origin  are given by $x_i=x_j=0$, for $1\leq i<j\leq \tau$. If $\lambda_j/ \lambda_i<0$, then $x_i=x_j=0$ is of nodal type and if $\lambda_j/\lambda_i>0$, it is of real saddle type.
\end{remark}

\subsection{Reduction of Singularities}  We say that a GH-foliated space $({\mathcal M},{\mathcal F})$ is {\em locally desingularized} if and only if every $p\in K$ is a simple point. Assume that  $({\mathcal M},{\mathcal F})$ is locally desingularized. By the description of simple points in \cite{Can-C,Can}, we know that the singular locus $\operatorname{Sing}({\mathcal F})$ is the union of a finite family $
\{\Gamma_i\}
$ of nonsingular codimension one closed analytic subspaces  of $(M,K)$, having locally normal crossings with $\mathcal M$. We say that
$({\mathcal M},{\mathcal F})$  is {\em desingularized} when the family
$
\{\Gamma_i\}_{i=1}^\ell
$
 has normal crossings with $\mathcal M$.

Consider a GH-foliated space $({\mathcal M},{\mathcal F})$, not necessarily desingularized, and a point $p\in M$. We say that $\mathcal F$ and $E$ {\em have normal crossings at $p$} if, and only if $p\notin\operatorname{Sing}({\mathcal F})$ and the union $E\cup H$ has normal crossings at $p$, where $H$ is the only germ of invariant hypersurface for $\mathcal F$ through $p$. The {\em adapted singular locus } $\operatorname{Sing}({\mathcal F}, E)$ is defined by
$$
\operatorname{Sing}({\mathcal F}, E)=\{p\in M; \; {\mathcal F}\text{ and } E
 \text{ do not have normal crossings at } p \}.
$$
It is a closed analytic set of $M$ of codimension $\geq 2$. Moreover, we have that
$\operatorname{Sing}({\mathcal F})\subset \operatorname{Sing}({\mathcal F},E)$. Let us note that in the case that $({\mathcal M},{\mathcal F})$ is desingularized, we have
$\operatorname{Sing}({\mathcal F},E)=\operatorname{Sing}({\mathcal F})$. Note also that any analytic subset of $\operatorname{Sing}({\mathcal F},E)$ is invariant for $\mathcal F$.

We say that a standard blow-up $\pi:({\mathcal M}',{\mathcal F}')\rightarrow ({\mathcal M},{\mathcal F})$ is {\em admissible} when the center is contained in the adapted singular locus of $({\mathcal M}, {\mathcal F})$. An {\em admissible transformation} is a finite composition of admissible blow-ups.

\begin{definition}  A {\em reduction of singularities} of a GH-foliated space $({\mathcal M},{\mathcal F})$ is any admissible transformation
$({\mathcal M}',{\mathcal F}')\rightarrow ({\mathcal M},{\mathcal F})$  such that $({\mathcal M}',{\mathcal F}')$ is desingularized.
\end{definition}

The problem of the existence of reduction of singularities for a GH-foliated space is the same one as the problem of reduction of singularities of the corresponding invariant hypersurfaces.  The key observation for this is the characterization of simple corners given in Proposition \ref{pro:simplecorner}. The precise statement is the following one:
\begin{proposition} Consider a GH-foliated space $({\mathcal M}_0,{\mathcal F}_0)$, where ${\mathcal M}_0=({\mathbb C}^n,E^0;\{0\})$. The list
$
{\mathcal L}_0=\{H_1^0,H_2^0,\ldots, H^0_k\}
$
of invariant irreducible hypersurfaces is non-empty and finite. Moreover, any reduction of singularities
$
\pi:({\mathcal M},{\mathcal L})\rightarrow ({\mathcal M}_0,{\mathcal L}_0)
$
 induces a reduction of singularities
$\pi:({\mathcal M},{\mathcal F})\rightarrow ({\mathcal M}_0,{\mathcal F}_0)$.
\end{proposition}
\begin{proof} See \cite{Fer-M}.
\end{proof}
As a consequence of Remark \ref{rk:reductionoflists}, there is at least one reduction of singularities for every GH-foliated space $({\mathcal M}_0,{\mathcal F}_0)$, where ${\mathcal M}_0=({\mathbb C}^n,E^0;\{0\})$.
\subsection{Invariant Hypersurfaces and Partial Separatrices}  Following \cite{Can, Can-C,Can-RV-S}, we recall here the description of the set of invariant hypersurfaces of a desingularized GH-foliated space $({\mathcal M},{\mathcal F})$, in terms of the so called ``partial separatrices''.

 Let $\Gamma$ be an irreducible component of $\operatorname{Sing}({\mathcal F})$. We recall that the generic points of $\Gamma$ have dimensional type two. Hence, we can consider the {\em generic transversal type} of $\mathcal F$ at $\Gamma$, defined by the two dimensional foliated space $(\Delta, \Delta\cap E, {\mathcal F}\vert_\Delta)$, where $\Delta$ is a two dimensional non singular germ, transversal to $\operatorname{Sing}({\mathcal F})$ at a generic point. The generic transversal type does not depend on the choice of the particular two dimensional section. Moreover, the two dimensional foliated space $(\Delta, \Delta\cap E, {\mathcal F}\vert_\Delta)$ is a germ at a simple point.
 \begin{definition}
 Consider an irreducible component $\Gamma$ of \/ $\operatorname{Sing}({\mathcal F})$. It  is of {\em trace type} when it is contained in a single irreducible component of $E$ and it is of {\em generic corner type} when it is contained in two irreducible components of $E$.
 \end{definition}
 \begin{remark}
  \label{rk:tracetypepoints}
  If $\Gamma$ a trace type irreducible component of $\operatorname{Sing}({\mathcal F})$, then any point of $\Gamma$ is a trace type singular point of $({\mathcal M},{\mathcal F})$, see Definition \ref{def:simplecorner}.
 This is equivalent to say that all the points in $\Gamma$ are trace type singular points. Moreover, any trace type singular point is contained in at least one trace type irreducible component of $\operatorname{Sing}({\mathcal F})$. Next example illustrates the situation. Take $(M,K)=({\mathbb C}^3,0)$, the divisor $E$ given by  $xy=0$ and let $\mathcal F$ be defined by
 $
 zdx/x+zdy/y+dz=0
 $; then the singular locus is given by
 $$
 (x=y=0)\cup (x=z=0)\cup (y=z=0),
 $$
 where the origin is a trace type singular point, the curves $x=z=0$  and $y=x=0$ are the trace type irreducible components of the singular locus and $x=y=0$ is an irreducible component of generic corner type.
\end{remark}

 Let us introduce some notations:

 We denote by $\operatorname{Sing}_{\mathcal M}^{\mathcal F}$ the set of irreducible components of $\operatorname{Sing}({\mathcal F})$. For any subset ${\mathcal A}$ of $\operatorname{Sing}_{\mathcal M}^{\mathcal F}$ we denote by $\vert{\mathcal A}\vert$ the {\em support} of $\mathcal A$, defined by
$$
\vert{\mathcal A}\vert=\cup \{\Gamma;\Gamma\in {\mathcal A}\}.
$$
For instance, we have $\vert \operatorname{Sing}_{\mathcal M}^{\mathcal F}\vert=\operatorname{Sing}({\mathcal F})$.

We denote by $\operatorname{Trace}_{\mathcal M}^{\mathcal F}$ the set of $\Gamma\in \operatorname{Sing}^{\mathcal F}_{\mathcal M}$ that are of trace type. In view of Remark  \ref{rk:tracetypepoints}, we have that $\vert \operatorname{Trace}_{\mathcal M}^{\mathcal F}\vert$ is the set of trace type singular points of $({\mathcal M},{\mathcal F})$.

 \begin{definition} A non empty subset  $\Sigma\subset \operatorname{Trace}_{\mathcal M}^{\mathcal F}$ is a {\em partial separatrix $\Sigma$ of $({\mathcal M},{\mathcal F})$} if and only if
 $\vert\Sigma\vert$ is a connected component of \/ $\vert \operatorname{Trace}_{\mathcal M}^{\mathcal F}\vert$.
 \end{definition}
The following properties, explained in \cite{Can-C}, show the relationship between partial separatrices and invariant hypersurfaces:
\begin{enumerate}
\item  Given a trace type simple singular point $p\in K$, there is a unique germ of invariant hypersurface $(H,p)$ not contained in $E$. Moreover we have that $$
 (H,p)\cap E=(\vert \operatorname{Trace}_{\mathcal M}^{\mathcal F}\vert,p).
 $$
 \item We can glue these local invariant hypersurfaces to obtain a natural bijection
$$
H\leftrightarrow \Sigma, \text{ with } H\cap E=\vert \Sigma\vert,
$$
between the set $\operatorname{Hyp}({\mathcal M},{\mathcal F})$ of closed irreducible invariant hypersurfaces of $({\mathcal M},{\mathcal F})$ not contained in $E$ and the set
$\operatorname{Par}({\mathcal M},{\mathcal F})$ of partial separatrices.
\end{enumerate}

 \begin{remark} Let $({\mathcal M},{\mathcal F})\rightarrow ({\mathcal M}_0,{\mathcal F}_0)$ be a reduction of singularities of $({\mathcal M}_0,{\mathcal F}_0)$, where ${\mathcal M}_0=({\mathbb C}^n,E^0; \{0\})$. By Grauert's Direct Image Theorem, there is a bijection between  $\operatorname{Hyp}({\mathcal M},{\mathcal F})$  and the set of irreducible germs of invariant hypersurfaces of ${\mathcal F}_0$, not contained in $E^0$. Hence, this last set is also faithfully represented by the set $\operatorname{Par}({\mathcal M},{\mathcal F})$ of partial separatrices.
 \end{remark}

\section{Nodal Separating Blocks}
We introduce here the definition and first properties of the {\em nodal separating blocks}
for a desingularized GH-foliated space $({\mathcal M},{\mathcal F})$.
The nodal separating blocks are the structures that may support  ``global barriers'' separating the space of leaves.

We define these objects under the assumption that $({\mathcal M},{\mathcal F})$ is a desingularized GH-foliated space; but, in order to assure that they separate the divisor into convenient connected components, we will need to enlarge the reduction of singularities to obtain the so-called {\em nodally reduced GH-foliated spaces}. In Subsection \ref{Nodally Desingularized Foliated Spaces} we show that this is always possible.

An element $\Gamma\in \operatorname{Sing}^{\mathcal F}_{\mathcal M}$ is of {\em generic nodal type} if and only if the generic transversal type of $({\mathcal M},{\mathcal F})$ at $\Gamma$ is a nodal singularity. We denote by $\operatorname{Nod}_{\mathcal M}^{\mathcal F}$ the set of $\Gamma\in \operatorname{Sing}_{\mathcal M}^{\mathcal F}$ that are of generic  nodal type.

\begin{remark}
\label{rk:nodalsingularpoints}
Not all the points of a given
$\Gamma\in \operatorname{Nod}_{\mathcal M}^{\mathcal F}$ are of nodal type. We only ask the property for the points with dimensional type equal to two. For instance, consider the space  ${\mathbb C}^3$, take $E=(xyz=0)$ and $\mathcal F$ the foliation given by $\eta=0$ with
$$
\eta=dx/x-\lambda dy/y+\mu dz/z,\quad  \lambda\in {\mathbb R}_{\geq 0}\setminus {\mathbb Q},\quad \mu\notin {\mathbb R}.
$$
The origin is not a nodal type singularity, but $x=y=0$ is an element of $\operatorname{Nod}^{\mathcal F}_{\mathcal M}$.
\end{remark}
We say that $\Gamma\in \operatorname{Nod}_{\mathcal M}^{\mathcal F}$ is {\em uninterrupted} if and only if all the points in $\Gamma$ are of nodal type.  We denote by $\operatorname{Nod}_{\mathcal M}^{*{\mathcal F}}$ the set of uninterrupted nodal components.
\begin{definition} A {\em nodal separating block} ${\mathcal B}$ for $({\mathcal M},{\mathcal F})$ is a subset ${\mathcal B}\subset \operatorname{Nod}_{\mathcal M}^{*{\mathcal F}}$ such that the support $\vert{\mathcal B}\vert$ is a connected component of $\vert
\operatorname{Nod}_{\mathcal M}^{\mathcal F}
\vert$.
 The {\em separator set} ${\mathcal S}_{\mathcal M}^{\mathcal F}$ is the union of the
nodal separating blocks.
\end{definition}

The supports $\vert B\vert$ of the nodal separating blocks $\mathcal B$ are the connected components of the support  $\vert {\mathcal S}_{\mathcal M}^{\mathcal F}\vert$ of the separator set ${\mathcal S}_{\mathcal M}^{\mathcal F}$. On the other hand, we have
$$
\vert {\mathcal S}_{\mathcal M}^{\mathcal F}\vert
\subset
 \vert \operatorname{Nod}_{\mathcal M}^{*{\mathcal F}}\vert
 \subset
 \{\text{nodal type singularities}\}
 \subset  \vert \operatorname{Nod}_{\mathcal M}^{{\mathcal F}}\vert,
$$
but the inclusions are not necessarily equalities. To see this, let us consider the example where $M={\mathbb C}^3$, $E=(xyz(x-1)=0)$ and $\mathcal F$ is given by $\omega=0$ with
$$
\omega=\frac{dx}{x}+\lambda\frac{dy}{y}-\mu\frac{dz}{z}-\alpha\frac{d(x-1)}{(x-1)},
$$
with $\lambda,\mu\in {\mathbb R}_{>0}\setminus {\mathbb Q}$, $\lambda/\mu\notin {\mathbb Q}$ and $\alpha\in {\mathbb C}\setminus {\mathbb R}$. In this case, we have that ${\mathcal S}_{\mathcal M}^{\mathcal F}$ is empty, and hence its support is also empty, moreover
\begin{eqnarray*}
\vert \operatorname{Nod}_{\mathcal M}^{*{\mathcal F}}\vert&=&(x=z=0)\\
\{\text{nodal type singularities}\}&=& (x=z=0)\cup ((y=z=0)\setminus\{(1,0,0)\})\\
\vert \operatorname{Nod}_{\mathcal M}^{{\mathcal F}}\vert&=& (x=z=0)\cup (y=z=0).
\end{eqnarray*}

\subsection{Nodally Reduced Foliated Spaces}
  \label{Nodally Desingularized Foliated Spaces}
  A GH-foliated space $({\mathcal M},{\mathcal F})$ is said to be {\em nodally reduced} if and only if is desingularized and each $\Gamma\in \operatorname{Nod}_{\mathcal M}^{\mathcal F}$ is of generic corner type. This property is a suitable condition to describe the separation of the ambient space by the nodal separating blocks. Indeed, in this situation, the nodal separating blocks provide locally a topological separation of the divisor $E$.

\begin{proposition}
\label{prop:completereduction} Let $({\mathcal M}',{\mathcal F}')$ be a desingularized GH-foliated space. There is an admisible transformation
$
({\mathcal M},{\mathcal F})\rightarrow ({\mathcal M}',{\mathcal F}')
$,
such that $({\mathcal M},{\mathcal F})$ is nodally reduced.
\end{proposition}
\begin{proof}
 Our objective is to obtain that
$\operatorname{Nod}_{{\mathcal M}}^{{\mathcal F}}\cap \operatorname{Trace}_{{\mathcal M}}^{{\mathcal F}} =\emptyset$.
Let $\Gamma$ be an element of $\operatorname{Nod}_{{\mathcal M}'}^{{\mathcal F}'}\cap \operatorname{Trace}_{{\mathcal M}'}^{{\mathcal F}'} $.
It has a transversal type of the form
$$
\lambda_{\Gamma} y\frac{dx}{x}-dy=0 , \quad E=(x=0),\, \lambda_{\Gamma}\in {\mathbb R}_{>0}\setminus {\mathbb Q}.
$$
Since $({\mathcal M}',{\mathcal F}')$ is desingularized, we know that $\Gamma$ is a center for an admissible blow-up. Let
$
({\mathcal M}_1,{\mathcal F}_1)\rightarrow  ({\mathcal M}',{\mathcal F}')
$
be the blow-up with center $\Gamma$. There is at most one $\Gamma_1\in \operatorname{Nod}_{{\mathcal M}_1}^{{\mathcal F}_1}\cap \operatorname{Trace}_{{\mathcal M}_1}^{{\mathcal F}_1}$, created over $\Gamma$. In this case $\lambda_{\Gamma_1}=\lambda_{\Gamma}-1$. After $k$ operations like this one, with $k-1 < \lambda\leq k$, we have that
 $$
\#(\operatorname{Nod}_{{\mathcal M}_k}^{{\mathcal F}_k}\cap \operatorname{Trace}_{{\mathcal M}_k}^{{\mathcal F}_k}) <
\#(\operatorname{Nod}_{{\mathcal M}'}^{{\mathcal F}'}\cap \operatorname{Trace}_{{\mathcal M}'}^{{\mathcal F}'}).
 $$
Thus, in finitely many steps, we obtain that $\operatorname{Nod}_{{\mathcal M}}^{{\mathcal F}}\cap \operatorname{Trace}_{{\mathcal M}}^{{\mathcal F}} =\emptyset$.
\end{proof}
\begin{lemma} Let $({\mathcal M},{\mathcal F})$ be a nodally reduced GH-foliated space. Then, all the points in $\vert\operatorname{Nod}_{\mathcal M}^{*\mathcal F}\vert$ are of corner type. ``A fortiori'', we also have that all the points in $\vert{\mathcal S}_{\mathcal M}^{\mathcal F}\vert$ are of corner type.
\end{lemma}
\begin{proof} Take a point $p\in \vert\operatorname{Nod}_{\mathcal M}^{*\mathcal F}\vert$, we know that it is a nodal type singularity. In view of the linearization property of nodal singularities, there are local coordinates $(x_i)_{i=1}^n$ at $p$ such that $E\subset (\prod_{i=1}^\tau x_i=0)$ and $\mathcal F$ is locally given by the logarithmic $1$-form
$$
\omega=\sum_{i=1}^s\lambda_i\frac{dx_i}{x_i}- \sum_{i=s+1}^\tau\lambda_i\frac{dx_i}{x_i},
$$
where $2\leq s \leq \tau-1$ and $ \lambda_i>0$, for $i=1,2,\ldots,\tau$. The nodal components of the singular locus through $p$ are
$
x_i=x_j=0
$, for
$1\leq i\leq s$ and $s+1\leq j\leq \tau$.
Since they are of corner type, we conclude that $(x_i=0)\subset E$ and $(x_j=0)\subset E$ for any $1\leq i\leq s$ and $s+1\leq j\leq \tau$. Hence $E=(\prod_{i=1}^\tau x_i=0)$, locally at $p$. Then $p$ is a corner type point.
\end{proof}
 Let $({\mathcal M},{\mathcal F})$ be a nodally reduced GH-foliated space. Then any partial separatrix $\Sigma$ is such that $\vert \Sigma\vert\cap \vert{\mathcal S}_{\mathcal M}^{\mathcal F}\vert=\emptyset$. To see this, it is enough to note that all the points in $\vert \Sigma\vert$ are trace type points, whereas the points of
$\vert{\mathcal S}_{\mathcal M}^{\mathcal F}\vert$ are of corner type. In other words, we have that $\vert \Sigma\vert$ is contained in one connected component $C_\Sigma$ of $E\setminus \vert{\mathcal S}_{\mathcal M}^{\mathcal F}\vert$. In terms of the associated invariant hypersurface $H_\Sigma$, let us note that $H_\Sigma\cap E= \vert\Sigma\vert$ and hence $H_\Sigma\cap E\subset C_\Sigma$.

\subsection{Connection Outside of the Separator Set}
\label{Connection Outside of the Separator Set}

Let $({\mathcal M},{\mathcal F)}$ be a nodally reduced GH-foliated space. In this subsection, we justify the definition of separator set by means of Proposition \ref{prop:separationbynodalblocks} below.

Recall the decomposition $E=\cup_{i\in I}E_i$ of the divisor $E$ into a finite union of nonsingular hypersurfaces. Denote by ${\mathcal H}_{\mathcal M}$ be the set whose elements are the subsets $J\subset I$, such that $E_J\ne \emptyset$, where $E_J=\cap _{j\in J}E_j$. In this way, we have a stratification of $M$ induced by $E$, whose strata are associated to the elements of ${\mathcal H}_{\mathcal M}$. Given $J\in {\mathcal H}_{\mathcal M}$, the corresponding stratum $S_J$ is
\begin{equation}
\label{eq:strata}
S_J=
E_J \setminus \cup \{ E_{J'}; \; J'\supset J,\, J'\ne J \}.
\end{equation}
It is connected in view of the assumptions on standard ambient spaces.

For any $k\geq 0$, we denote by ${\mathcal H}_{\mathcal M}(k)$ the set of $J\in {\mathcal H}_{\mathcal M}$ with exactly $k$-elements.
\begin{definition}
Given ${A}\subset {\mathcal H}_{\mathcal M}(2)$, we say that two irreducible components $E_i$ and $E_j$ of $E$ {\em are connected outside ${A}$} if, and only if, there is a finite sequence
$$
i=i_0,i_1,i_2,\ldots,i_r=j,
$$
such that $\{i_{k},i_{k+1}\}\in {\mathcal H}_{\mathcal M}(2)\setminus { A}$, for any $k=0,1,\ldots,r-1$.
\end{definition}
\begin{remark}
Once we take appropriate representatives of the germs $(M,K)$ and $(E,E\cap K)$, to say that $E_i$ and $E_j$ are connected outside $A\subset {\mathcal H}_{\mathcal M}(2)$ is equivalent to say that we can topologically connect the points in the strata of $S_{\{i\}}\subset E_i$ with the points in  $S_{\{j\}}\subset E_j$ by means of pathes contained in $E\setminus \cup_{J\in A}E_J$.
\end{remark}

 Since any nodal irreducible component of $\operatorname{Sing}({\mathcal F})$ is of corner type,  we can identify $\operatorname{Nod}_{\mathcal M}^{\mathcal F}$ to a subset of ${\mathcal H}_{\mathcal M}(2)$.

\begin{proposition}
\label{prop:separationbynodalblocks} Let $({\mathcal M}, {\mathcal F})$ be a
nodally reduced foliated space and take two irreducible components $E_i$, $E_j$ of $E$. The following statements are equivalent:
 \begin{enumerate}
\item $E_i$ and $E_j$ are connected outside $\operatorname{Nod}_{\mathcal M}^{\mathcal F}$.
 \item $E_i$ and $E_j$ are connected outside $\mathcal{S}_{\mathcal M}^{\mathcal F}$.
 \end{enumerate}
\end{proposition}
\begin{proof} Since $\operatorname{Nod}_{\mathcal M}^{\mathcal F}\supset \mathcal{S}_{\mathcal M}^{\mathcal F}$, we have that (1) implies (2). Now, it is enough to show that
if $\Gamma=E_k\cap E_\ell\in \operatorname{Nod}_{\mathcal M}^{\mathcal F}\setminus {\mathcal S}_{\mathcal M}^{\mathcal F}$, then $E_k$ and $E_\ell$ are connected outside $\operatorname{Nod}_{\mathcal M}^{\mathcal F}$.

Assume first that there is a not nodal point $p\in \Gamma$. The point $p$ cannot be a real saddle and hence it is not of real type. This means that in suitable local coordinates $x_1,x_2,\ldots,x_n$ at $p$, we have $E_k=(x_1=0)$, $E_\ell=(x_2=0)$ and the residual vector
$$
(\lambda_1,\lambda_2,\lambda_3,\ldots,\lambda_\tau)
$$
satisfies $\lambda_2/\lambda_1<0$, $\lambda_3/\lambda_1\notin {\mathbb R}$. Hence $\lambda_3/\lambda_2\notin {\mathbb R}$. If
$
E_r=(x_3=0)
$,
 we see that
 $p\in E_r$ and both  $E_k\cap E_r$ and $E_r\cap E_\ell$ are not generically nodal. Then $E_k$ and $E_\ell$ are connected outside $\operatorname{Nod}_{\mathcal M}^{\mathcal F}$.

 In a general way, let $\mathcal B\subset \operatorname{Nod}_{\mathcal M}^{\mathcal F}$ be the ``nodal block'' such that $\Gamma\in \mathcal B$. By this, we mean that $\vert{\mathcal B}\vert$ is the connected component of $\vert \operatorname{Nod}_{\mathcal M}^{\mathcal F}\vert$ containing $\Gamma$.
  We know that there is a point $p\in \vert \mathcal B\vert$ that is not of nodal type: otherwise $\mathcal B$ would be a nodal separating block and hence $\Gamma$ an element of the separator set ${\mathcal S}_{\mathcal M}^{\mathcal F}$. We find a finite sequence
$$
\Gamma=\Gamma_0\ni p_{01}\in \Gamma_1\ni p_{12}\in \Gamma_2\ni\cdots \in \Gamma_s\ni p,
$$
where $\Gamma_u\in \operatorname{Nod}_{\mathcal M}^{\mathcal F}$ for $u=0,1,\ldots,s$ and the points $p_{u,u+1}$ are nodal points for $u=0,1,\ldots,s-1$. If $s=0$, we are done.  Assume that $s\geq 1$. Since $p_{01}$ is a nodal point,  we can write
$\Gamma_1=E_{k_1}\cap E_{\ell_1}$, where $E_k\cap E_{k_1}$ and $E_{\ell}\cap E_{\ell_1}$ are not generically nodal (to see this, take $E_{k_1}$ of the ``same sign'' as $E_k$ and $E_{\ell_1}$ of the ``same sign'' as $E_\ell$). Working by induction on $s$, we can connect $E_{k_1}$ and $E_{\ell_1}$ outside  $\operatorname{Nod}_{\mathcal M}^{\mathcal F}$. We end since $E_k$ and $E_{k_1}$, respectively $E_\ell$ and $E_{\ell_1}$, are connected outside  $\operatorname{Nod}_{\mathcal M}^{\mathcal F}$.
\end{proof}
As an immediate consequence, we obtain the following corollary:

\begin{corollary}
\label{cor:separationbynodalblocks} Let $({\mathcal M}, {\mathcal F})$ be a
nodally reduced foliated space. There is a bijection
$$
\left\{
\begin{array}{c}
\mbox{ connected components of } \\
E\setminus \vert\operatorname{Nod}_{\mathcal M}^{\mathcal F}\vert
\end{array}
\right
\}
\leftrightarrow
\left\{
\begin{array}{c}
\mbox{ connected components of } \\
E\setminus \vert{\mathcal S}_{\mathcal M}^{\mathcal F}\vert
\end{array}
\right
\}
$$
such that a connected component $C$ of $E\setminus \vert\operatorname{Nod}_{\mathcal M}^{\mathcal F}\vert$ corresponds to the only connected component $C'$ of $E\setminus \vert{\mathcal S}_{\mathcal M}^{\mathcal F}\vert$ that contains $C$.
\end{corollary}

\section{The complement of Nodal Separating  blocks}

Let
$({\mathcal M},{\mathcal F})$ be a nodally reduced GH-foliated space.  In this section we count the number of connected components of $
E \setminus \vert {\mathcal S}_{\mathcal M}^{\mathcal F}\vert
$ in terms of the number of nodal separating blocks of $({\mathcal M}, {\mathcal F})$.
 We do it when $\mathcal M$ is {\em combinatorially simply connected}. In next section we prove that  this  property holds for any ambient space $\mathcal M$ obtained from any standard transformation of $({\mathbb C}^n,0)$.

For the above purpose, we work in an abstract and combinatorial way, based on the combinatorial strata structure {\em ${\mathcal H}_{\mathcal M}$} associated to ${\mathcal M}$. The definition of combinatorial strata structure is essentially the same one of  abstract  simplicial complex (see \cite{Koz}); it has also been used for describing the combinatorial part of a desingularization procedure in\cite{Moli}, under the name of ``support fabric''. In low dimensional cases, we recover the classical idea of the dual graph of a divisor.

\subsection{Combinatorial Strata Structures}

Let $I$ be a finite set of indices and denote by ${\mathcal P}(I)$ the set of subsets of $I$. Define the topology on ${\mathcal P}(I)$ whose open sets are the unions of sets of the form ${\mathcal P}(J)$, where $J\subset I$.
A {\em combinatorial strata structure, with minimal set of indices $I$,} is an open subset ${\mathcal H}\subset{\mathcal P}(I)$ such that $\{i\}\in {\mathcal H}$, for any $i\in I$. The elements of $\mathcal H$ are called {\em ${\mathcal H}$-strata} or simply {\em strata} if no confusion arises. The {\em maximal codimension of $\mathcal H$} is the maximal number of elements of a stratum in $\mathcal H$. The closure $\operatorname{Cl}_{\mathcal H}(A)$ of a subset $A\subset {\mathcal H}$ is $\operatorname{Cl}_{\mathcal H}(A)=\cup_{J\in A}\operatorname{Cl}_{\mathcal H}(\{J\})$, where $\operatorname{Cl}_{\mathcal H}(\{J\})=\{J'\in {\mathcal H};\; J'\supset J\}$.

\begin{example} If ${\mathcal M}=(M,E;K)$ is a standard ambient space, then ${\mathcal H}_{\mathcal M}$, as defined in Subsection \ref{Connection Outside of the Separator Set}, is a combinatorial strata structure.
\end{example}

Let $\mathcal H$ be a combinatorial strata structure  with minimal set of indices $I$.

For any $k\geq 0$, we denote by ${\mathcal H}(k)$ the set of $J\in {\mathcal H}$ such that $J$ has exactly $k$ elements. A {\em $k$-path $\gamma$ in ${\mathcal H}$ of length $s\geq 0$} is a finite sequence
$
\gamma=(J_0,J_{1}, \ldots,J_{s})
$,
such that $J_t\in {\mathcal H}(k)$, for $t=0,1,\ldots,s$ and  $J_{t-1}\cup J_t\in {\mathcal H}(k+1)$ ,
for $t=1,2,\ldots,s$. We say that {\em $\gamma$ joins $J_0$ and $J_s$}.
The {\em support $\mbox{\rm Sop}(\gamma)$  and the subsupport $\mbox{\rm Sub}(\gamma)$ of $\gamma$} are respectively given by
\begin{equation*}
\begin{array}{lcccc}
\mbox{\rm Sop}(\gamma)&=&\{J_0,J_1,\ldots,J_s\}&\subset & {\mathcal H}(k)\\
\mbox{\rm Sub}(\gamma)&=&(J_0\cup J_1,J_1\cup J_2,\ldots,J_{s-1}\cup J_s)&\in& {\mathcal H}(k+1)^{s}.
\end{array}
\end{equation*}
Given two $k$-paths $\gamma$ and $\gamma'$ in $\mathcal H$, joining respectively $J$ with $J'$ and $J'$ with $J''$, the {\em  composition}
$\gamma*\gamma'$ is a $k$-path in $\mathcal H$, joining $J$ and $J''$ and defined in a evident way.
The {\em reverse $\gamma^{-1}$} is defined  by reversing the ordering in the sequence $\gamma$.

 A subset $A\subset {\mathcal H}(k)$ is {\em $k$-connected in ${\mathcal H}$} if and only if for any $J,J'\in A$ there is a $k$-path $\gamma$ in $\mathcal H$ joining $J$ and $J'$ with support contained in $A$.
The union of two $k$-connected sets $A_1,A_2\subset {\mathcal H}(k)$ in $\mathcal H$ such that $A_1\cap A_2\ne\emptyset$ is also $k$-connected in $\mathcal H$. A {\em $k$-connected component of $A\subset {\mathcal H}(k)$ in ${\mathcal H}$} is any nonempty subset of $A$ that is $k$-connected in $\mathcal H$ and maximal with this property. Any $A\subset {\mathcal H}(k)$ has a unique partition into $k$-connected components in $\mathcal H$.

Let us simplify the terminology as follows: we say that {\em ${\mathcal H}$ is $k$-connected} if and only if ${\mathcal H}(k)$ is $k$-connected in $\mathcal H$; the {\em $k$-connected components of $\mathcal H$} are the $k$-connected components of ${\mathcal H}(k)$ in $\mathcal H$.
\begin{remark}
\label{rk:connected components}
Concerning the $1$-conectedness, let us precise the concept of connected components.  There are a unique partition $I=\cup_{\lambda \in \Lambda}^s{I_\lambda}$ and a decomposition
$$
{\mathcal H}=\cup_{\lambda\in \Lambda}{\mathcal H}_\lambda,
$$
in such a way that each ${\mathcal H}_\lambda$ is a $1$-connected combinatorial strata structure with minimal set of indices $I_\lambda$ and ${\mathcal H}_\lambda\cap {\mathcal H}_{\lambda'}=\{\emptyset\}$ for $\lambda\ne \lambda'$. We say that the ${\mathcal H}_\lambda$ are the {\em connected components of ${\mathcal H}$}. To see this, we take the set  $I_\lambda$ containing $i\in I$ to be the set of indices $j\in I$ such that $\{i\}$ and $\{j\}$ are connected by a $1$-path. We define ${\mathcal H_\lambda}=\{J\in {\mathcal H};\, J\subset I_{\lambda}\}$.
\end{remark}

An {\em elementary homotopic pair $(\epsilon,\epsilon')$ in $\mathcal H$} is a pair
of $1$-paths in $\mathcal H$ such that, up to reordering of $\epsilon,\epsilon'$, we have that either $\epsilon=\epsilon'$, or $\epsilon=(\{i_1\},\{i_2\}, \{i_1\})$ and $\epsilon'=(\{i_1\})$ or, finally, we have
$$
\epsilon=(\{i_1\},\{i_2\})\text{ and } \epsilon'=(\{i_1\},\{i_3\},\{i_2\}),
$$
where
   $
   (\{i_1,i_3\},\{i_3,i_2\})
   $
   is a $2$-path in $\mathcal H$.

Take two $1$-paths $\gamma_1,\gamma_2$ in $\mathcal H$. We say that $\gamma_2$ has been obtained from $\gamma_1$ by an {\em elementary homotopy} (that we denote $\gamma_1\rightsquigarrow\gamma_2$) if there is an elementary homotopic pair $(\epsilon_1,\epsilon_2)$ and $1$-paths $\delta,\rho$ in $\mathcal H$, such that
$$
\gamma_1=\delta*\epsilon_1*\rho; \quad  \gamma_2=\delta*\epsilon_2*\rho.
$$
Consider a subset $A\subset {\mathcal H}(1)$. Two
$1$-paths $\gamma,\gamma'$ in ${\mathcal H}$ are {\em homotopically equivalent in ${\mathcal H}$ with support in $A$}, if there is a finite sequence of elementary homotopies
 $$
 \gamma=\gamma_0\rightsquigarrow\gamma_1\rightsquigarrow\cdots\rightsquigarrow\gamma_s=\gamma'
 $$
 such that $\mbox{\rm Sop}(\gamma_\ell)\subset A$ for $\ell=0,1,\ldots, s$.

A $1$-connected subset  $A\subset {\mathcal H}(1)$ is {\em simply connected in $\mathcal H$} if, and only if,
any two $1$-paths $\gamma,\gamma'$ in $\mathcal H$, with support in $A$ and joining the same strata, are homotopically equivalent in $\mathcal H$ with support in $A$. We say that $\mathcal H$ is {\em simply connected } if and only if $\mathcal H(1)$ is simply connected in $\mathcal H$.

\begin{remark}
 \label{rk:tresdimensionalcase}
 Let us consider a combinatorial strata structure $\mathcal H\subset {\mathcal P}(I)$ with minimal set of indices $I$. The subset ${\mathcal H}_3\subset {\mathcal H}$ defined by
\begin{equation}
\label{eq:redtres}
{\mathcal H}_3=\{\emptyset\}\cup {\mathcal H}(1)\cup {\mathcal H}(2)\cup {\mathcal H}(3)
\end{equation}
is an open set of $\mathcal H$ and hence of ${\mathcal P}(I)$. Then ${\mathcal H}_3$ is a
combinatorial strata structure with minimal set of indices $I$. The $1$-paths and the $2$-paths in ${\mathcal H}_3$ are the same ones as  in ${\mathcal H}$. In particular, a subset $A\subset {\mathcal H}(1)$,  is $1$-connected, respectively simply connected,  in ${\mathcal H}$, if and only if it is  $1$-connected, respectively simply connected,  in ${\mathcal H}_3$.
Moreover  a subset $B\subset {\mathcal H}(2)={\mathcal H}_3(2)$ is $2$-connected in ${\mathcal H}$ if and only if it is $2$-connected in ${\mathcal H}_3$.
Hence the $1$-connectedness and the simple connectedness of $A\subset {\mathcal H}(1)$, as well as the $2$-connectedness of $B\subset {\mathcal H}(2)$, are questions ``in maximal codimension three''.
\end{remark}

\subsection{Abstract Datum of Nodal Strata}
A {\em datum of nodal strata $({\mathcal N},\{P_J\}_{J\in{\mathcal N}})$ in $\mathcal H$} is given by a subset ${\mathcal N}\subset {\mathcal H}$ and a partition $P_J=\{J_{+},J_{-}\}$ of any $J\in {\mathcal N}$ in two nonempty elements, satisfying the following property
\begin{quote}
``For any $J\in {\mathcal N}$ and $J'\subset J$ we have that $J'\in {\mathcal N}$ if and only if $J'\cap J_+\ne\emptyset\ne J'\cap J_{-}$. In this case $P_{J'}=\{J'\cap J_+, J'\cap J_-\}$.''
\end{quote}
 The set ${\mathcal N}^*$ of {\em uninterrupted nodal strata} is the subset ${\mathcal N}^*\subset {\mathcal N}$ given by the elements $J\in \mathcal N$ such that
$
\mbox{\rm Cl}_{\mathcal H}(\{J\})\subset {\mathcal N}
$.

\begin{remark} We have a topology on the finite set ${\mathcal H}$, hence the ``opposite'' topology is well defined, where the ``opposite open sets'' are the closed sets.
In this setting, we see that ${\mathcal N}^*$ is the biggest closed set of $\mathcal H$ contained in ${\mathcal N}$ (the interior of $\mathcal N$ for the opposite topology).
\end{remark}

\begin{remark} Let ${\mathcal H'}\subset {\mathcal H}$ be an open set. Then ${\mathcal N}'$, given by ${\mathcal N}'={\mathcal N}\cap {\mathcal H}'$, with the same partitions as $\mathcal N$, gives a datum of nodal strata for ${\mathcal H}'$.
\end{remark}

\begin{example}  Let $({\mathcal M},{\mathcal F})$ be a nodally reduced GH-foliated space.  We define a datum of nodal strata ${\mathcal N}_{{\mathcal M},{\mathcal F}}$ in ${\mathcal H}_{\mathcal M}$ as follows. Given $J\in {\mathcal H}_{\mathcal M}$, we say that $J\in {\mathcal N}_{{\mathcal M},{\mathcal F}}$ if and only if there is a nodal corner point $p\in S_J$ (this is equivalent to say that the corner points of $S_J$ are of nodal type), see Remark \ref{rk:strata}. In this case, the foliation is given by $\eta=0$ locally at $p$, where $J=J_+\cup J_-$ and
$$
\eta=\sum_{j\in J_+}\lambda_j\frac{dx_j}{x_j}- \sum_{j\in J_-}\lambda_j\frac{dx_j}{x_j},
\quad \lambda_j\in {\mathbb R}_{>0}; \; E=(\prod_{j\in J}x_j=0).
$$
Let us note that ${\mathcal N}_{{\mathcal M},{\mathcal F}}(2)=\operatorname{Nod}_{\mathcal M}^{\mathcal F}$, where ${\mathcal N}_{{\mathcal M},{\mathcal F}}(2)= {\mathcal N}_{{\mathcal M},{\mathcal F}}\cap {\mathcal H}_{\mathcal M,\mathcal F}(2)$.
\end{example}

Let  $({\mathcal N},\{P_J\}_{J\in{\mathcal N}})$ be a datum of nodal strata in $\mathcal H$. A {\em nodal block of ${\mathcal N}$ in $\mathcal H$} is by definition any $2$-connected component of ${\mathcal N}(2)$ in ${\mathcal H}$, where ${\mathcal N}(2)={\mathcal N}\cap {\mathcal H}(2)$.
\begin{lemma}
  \label{lema:sepblocks}
  For any nodal block $\mathcal B$ of $\mathcal N$ in $\mathcal H$, the following properties are equivalent:
 \begin{equation*}
 \operatorname{a) } \mathcal B \subset {\mathcal N}^*
 .\quad\operatorname{b) } \mbox{\rm Cl}_{\mathcal H}(\mathcal B)\subset {\mathcal N}
 .\quad \operatorname{c) } \mbox{\rm Cl}_{\mathcal H}(\mathcal B)\subset {\mathcal N}^*.
 \end{equation*}
\end{lemma}
\begin{proof} We see that ``c) implies b)''. Let us see that ``b) implies a)'': given $J\in {\mathcal B}$, we know that $\operatorname{Cl}_{\mathcal H}(\{J\})\subset {\mathcal N}$ and hence $J\in {\mathcal N}^*$. To see that ``a) implies c)'' is enough to note that ${\mathcal N}^*$ is closed.
\end{proof}
We
say that a nodal block $\mathcal B$ is a {\em nodal separating block} when it satisfies the equivalent properties in Lemma \ref{lema:sepblocks}. The {\em separator set ${\mathcal S}_{\mathcal H}^{\mathcal N}$ } is the union of the nodal separating blocks.
Let us note that
$$
{\mathcal S}_{\mathcal H}^{\mathcal N}={\mathcal H}(2)\cap\operatorname{Cl}_{\mathcal H}({\mathcal S}_{\mathcal H}^{\mathcal N})
\subset
\operatorname{Cl}_{\mathcal H}({\mathcal S}_{\mathcal H}^{\mathcal N})
\subset {\mathcal N}^*\subset {\mathcal N}.
$$
Let us put
$
{\mathcal H}_{{\mathcal S}_{\mathcal H}^{\mathcal N}}={\mathcal H}\setminus \operatorname{Cl}_{\mathcal H}({\mathcal S}_{\mathcal H}^{\mathcal N})
$; it
is a combinatorial strata structure, with $I$ as minimal set of indices.

Next Theorem \ref{teo:componentsconexas} gives the number of $1$-connected components of ${\mathcal H}_{{\mathcal S}_{\mathcal H}^{\mathcal N}}$ in terms of the number of nodal separating blocks.

\begin{theorem}
\label{teo:componentsconexas}
Let ${\mathcal H}$ be a simply connected combinatorial strata structure endowed with a   datum
$({\mathcal N},\{P_J\}_{J\in {\mathcal N}})$
of nodal strata and denote by $n$ the number of nodal separating blocks of $\mathcal N$ in $\mathcal H$. Then, the number of \/$1$-connected components of ${\mathcal H}_{{\mathcal S}_{\mathcal H}^{\mathcal N}}$ is equal to $n+1$.
\end{theorem}

\begin{remark} We have that Theorem \ref{teo:componentsconexas} implies Theorem \ref{teo:numerodecomponents}. To see this, note that the $1$-connected components of ${\mathcal H}_{{\mathcal S}_{\mathcal H}^{\mathcal N}}$ depend only on the sets ${\mathcal H}_{{\mathcal S}_{\mathcal H}^{\mathcal N}}(1)={\mathcal H}(1)$ and
${\mathcal H}_{{\mathcal S}_{\mathcal H}^{\mathcal N}}(2)={\mathcal H}(2)
\setminus {\mathcal S}_{\mathcal H}^{\mathcal N}
$.
\end{remark}
\subsection{Parity Results}
 \label{Parity Results}
 We start here the proof of Theorem \ref{teo:componentsconexas}. In this Subsection, we give properties concerning a nodal separating block $\mathcal B$ of $\mathcal N$ in $\mathcal H$.

\begin{lemma}
\label{lemma:separacion}
Let $\mathcal B'$ be a nodal block with  $\mathcal B'\ne \mathcal B$. Then $
\mbox{\rm Cl}_{\mathcal H}(\mathcal B)\cap \mbox{\rm Cl}_{\mathcal H}(\mathcal B')=\emptyset
$.
\end{lemma}
\begin{proof}  Take $J_0\in \mbox{\rm Cl}_{\mathcal H}(\mathcal B)\cap
\mbox{\rm Cl}_{\mathcal H}(\mathcal B')$ to find a contradiction.
There are $J\in {\mathcal B}$ and $ J'\in {\mathcal B'}$ such that $J, J'\subset J_0$. Note that $J\ne J'$ since $\mathcal B\cap \mathcal B'=\emptyset$. Put $K=J\cup J'$. We have that $K\in {\mathcal H}$, since $K\subset J_0$. Moreover $K\in \mbox{\rm Cl}_{\mathcal H}(\{J\})$ and since $J\in {\mathcal N}^*$, we deduce that $K\in {\mathcal N}$.

If $J\cap J'\ne \emptyset$, then $K\in {\mathcal H}(3)$.  Thus, we have a $2$-path
$
(J,K,J')
$
that joints $J$ with $J'$ with support in $\mathcal N$. This is not possible.

Suppose that  $J\cap J'=\emptyset$ and put $J=\{i_1,i_2\}$, $J'=\{i'_1,i'_2\}$. We have that
$K\in {\mathcal H}(4)$.
Assume up to a reordering that
$i_1,i'_1\in K_{+}$ and $i_2,i'_2\in K_{-}$. Then $\{i_1,i'_2\}\in {\mathcal N}$ and we have a $2$-path
$
(J,\{i_1,i'_2\}, J')
$
joining $J$ and $J'$, with support in $\mathcal N$. This is not possible.
\end{proof}
 Let $\gamma$ be a $1$-path in $\mathcal H$.  We denote by
$\kappa_{\mathcal B}(\gamma)$ the number of entries in $\operatorname{Sub}(\gamma)$ belonging to ${\mathcal B}$. That is $\kappa_{\mathcal B}(\gamma)$ denotes the number of ``times'' that $\gamma$ ``crosses through $\mathcal B$''.
 Given
$i\in I$, let $A_{\mathcal B}(i)$, respectively $B_{\mathcal B}(i)$, be the subset of ${\mathcal H}(1)$ whose elements are the strata $\{j\}\in {\mathcal H}(1)$ satisfying that there is a $1$-path $\gamma$ joining $\{i\}$ and $\{j\}$ in $\mathcal H$ with even $\kappa_{\mathcal B}(\gamma)$, respectively odd $\kappa_{\mathcal B}(\gamma)$. Since ${\mathcal H}(1)$ is $1$-connected in $\mathcal H$, we have that
$
{\mathcal H}(1)=A_{\mathcal B}(i)\cup B_{\mathcal B}(i)$,  for any $i\in I
$.
Moreover, by reversing and composing $1$-paths, we see that
$$
\begin{array}{ccccc}
A_{\mathcal B}(i')=A_{\mathcal B}(i)&\mbox{ and }& B_{\mathcal B}(i')=B_{\mathcal B}(i),&\mbox{ for any } \{i'\}\in A_{\mathcal B}(i). \\
A_{\mathcal B}(i')=B_{\mathcal B}(i)&\mbox{ and }& B_{\mathcal B}(i')=A_{\mathcal B}(i),&\mbox{ for any } \{i'\}\in B_{\mathcal B}(i).
\end{array}
$$
Note that $A_{\mathcal B}(i)\ne\emptyset$ since $\{i\}\in A_{\mathcal B}(i)$.  We also have that $B_{\mathcal B}(i)\ne\emptyset$, indeed, we can choose $K=\{j',j''\}\in \mathcal B$ and we either have that $\{j'\}\in B_{\mathcal B}(i)$ or $\{j''\}\in B_{\mathcal B}(i)$.

\begin{lemma} Let $\gamma$ and $\gamma'$ be two $1$-paths in $\mathcal H$ joining  $\{i\}$ and $\{j\}$. Then, the number  $\kappa_{\mathcal B}(\gamma)-\kappa_{\mathcal B}(\gamma')$ is even. In particular
$A_{\mathcal B}(i)\cap B_{\mathcal B}(i)=\emptyset$, for any $i\in I$.
\end{lemma}
\begin{proof} Since ${\mathcal H}$ is simply connected, there is a finite sequence
of elementary homotopies that transforms $\gamma$ in $\gamma'$.
It is enough to consider the case when $\gamma\rightsquigarrow\gamma'$ is an elementary homotopy. That is, we can assume that
$$
\gamma=\delta*\epsilon*\rho; \quad  \gamma'=\delta*\epsilon'*\rho,
$$
where $(\epsilon,\epsilon')$ is an elementary homotopic pair. Since
$\kappa_{\mathcal B}(\gamma)=\kappa_{\mathcal B}(\delta)+\kappa_{\mathcal B}(\epsilon)+\kappa_{\mathcal B}(\rho)$,
it is enough to show that $\kappa_{\mathcal B}(\epsilon)-\kappa_{\mathcal B}(\epsilon')$ is an even number. If $\epsilon=\epsilon'$ we are done. Otherwise, we have two possibilities:

{\em Case: $
\epsilon=(\{i_1\})$, $\epsilon'=(\{i_1\},\{i_2\},\{i_1\})
$}. Note that $\kappa_{\mathcal B}(\epsilon)=0$. Moreover, we have $ \kappa_{\mathcal B}(\epsilon')=0$, if $\{i_1,i_2\}\notin {\mathcal B}$ and $ \kappa_{\mathcal B}(\epsilon')=2$, if $\{i_1,i_2\}\in {\mathcal B}$. Hence $\kappa_{\mathcal B}(\epsilon)-k_{\mathcal B}(\epsilon')$ is an even number.

{Case: $\epsilon=(\{i_1\},\{i_2\})$ and $\epsilon'=(\{i_1\},\{i_3\},\{i_2\})$, where $
   (\{i_1,i_3\},\{i_3,i_2\})
   $
is a $2$-path in $\mathcal H$.}
Put $K=\{i_1,i_2,i_3\}\in {\mathcal H}(3)$ and consider the following possibilities:

$\bullet$ $\{i_1,i_2\}\in{\mathcal B}$.  In this case  $\kappa_{\mathcal B}(\epsilon)=1$.
We have that $K\in \operatorname{Cl}_{\mathcal H}({\mathcal B})$ and hence $K\in \mathcal N$. Noting that $i_1,i_2$ have not the same ``sign'', up to symmetry, we  assume $K_+=\{i_1,i_3\}$ and $K_-=\{i_2\}$. We get that $\{i_2,i_3\}\in {\mathcal N}$ and $\{i_1,i_3\}\notin {\mathcal N}$; moreover $\{i_2,i_3\}$ is $2$-connected with $\{i_1,i_2\}$, then $\{i_2,i_3\}\in {\mathcal B}$. We deduce that $\kappa_{\mathcal B}(\epsilon')=1$.

$\bullet$
$\{i_1,i_2\}\in{\mathcal N}\setminus {\mathcal B}$. In this case $\kappa_{\mathcal B}(\epsilon)=0$. Then $K\in \mbox{\rm Cl}_{\mathcal H}(\{i_1,i_2\})$ and by Lemma \ref{lemma:separacion} we have that $K\notin \mbox{\rm Cl}_{\mathcal H}({\mathcal B})$. Since
$$
K\in \mbox{\rm Cl}_{\mathcal H}(\{i_1,i_3\})\cap  \mbox{\rm Cl}_{\mathcal H}(\{i_2,i_3\}),
$$
we deduce that $\{i_1,i_3\}\notin {\mathcal B}$ and $\{i_2,i_3\}\notin {\mathcal B}$. Thus $\kappa_{\mathcal B}(\epsilon')=0$.

$\bullet$ $\{i_1,i_2\}\notin{\mathcal N}$. In this case $\kappa_{\mathcal B}(\epsilon)=0$.  If $\kappa_{\mathcal B}(\epsilon')=0$, we are done. Assume $\kappa_{\mathcal B}(\epsilon')\geq 1$ and that  $\{i_1,i_3\}\in {\mathcal B}$. Then
$K\in \operatorname{Cl}_{\mathcal H}{\mathcal B}\subset \mathcal N$. Moreover $i_1,i_2$ have the same ``sign'' in $K$, thus  $K_{+}=\{i_1,i_2\}$ and $K_-=\{i_3\}$, this implies that $\{i_2,i_3\}\in {\mathcal N}$. Then $\{i_2,i_3\}\in {\mathcal B}$ since it is 2-linked with $\{i_1,i_3\}$. Thus $\kappa_{\mathcal B}(\epsilon')=2$.
\end{proof}
\begin{remark} Take $i,j\in I$. We have that $\{j\}\in A_{\mathcal B}(i)$ if and only if $A_{\mathcal B}(j)= A_{\mathcal B}(i)$. In the same way, we have that $\{j\}\in B_{\mathcal B}(i)$ if and only if $A_{\mathcal B}(j)= B_{\mathcal B}(i)$.
\end{remark}

\begin{lemma}
\label{lema:conexiondeA} Consider two distinct indices $i,j\in I$ such that $A_{\mathcal B}(i)=A_{\mathcal B}(j)$.
Assume that ${\mathcal W}_{ij}\ne \emptyset$, where
${\mathcal W}_{ij}$ is the set of $2$-paths $\omega=(K_0,K_1,\ldots,K_u)$ in $\mathcal H$, with support in ${\mathcal B}$
such that $i\in K_0$ and  $j\in K_{u}$. Then, there is a $1$-path $\beta$ in $\mathcal H$ joining $\{i\}$ and $\{j\}$ such that $J\notin {\mathcal N}$ for any entry $J$ in $\operatorname{Sub}(\gamma)$.
In particular, we have that $\kappa_{\mathcal B}(\beta)=0$.
\end{lemma}
\begin{proof} We proceed by induction on the minimal length $\ell_{ij}$ of an element $\omega$ of $\mathcal W_{ij}$. If $\ell_{ij}=0$, there is $K=\{i,j\}\in \mathcal B$ and this contradicts the fact that $A_{\mathcal B}(i)=A_{\mathcal B}(j)$.  Thus, the induction starts at  $\ell_{ij}=1$. In this case, there is a $2$-path $(K_0, K_1)$ with $K_0,K_1\in \mathcal B$ such that $i\in K_0$, $j\in K_1$ and $K_{01}=K_0\cup K_1\in {\mathcal H}(3)$.
Take $J=\{i,j\}$. Since $J\subset K_{01}$, we have that $J\in {\mathcal H}(2)$. Moreover $J\notin {\mathcal N}$, otherwise $j\in B_{\mathcal B}(i)$. Then, we can take the $1$-path $\beta=(\{i\},\{j\})$.

Let us prove the inductive step. Take $\omega=(K_0,K_1,\ldots,K_u)\in \mathcal W_{ij}$ of  length $u=\ell_{ij}>1$. We can write
$$
K_0=\{i,j'\},\;  K_1=\{j',j''\},\; K_{01}=K_0\cap K_1=\{i,j',j''\}.
$$
Note that $\{i,j\}\cap K_1=\emptyset$, by the minimality of $u$ and te fact that $u>1$.
Moreover, since $\{j'\}\in B_{\mathcal B}(i)\cap  B_{\mathcal B}(j'')$, then $B_{\mathcal B}(i)=B_{\mathcal B}(j'')$ and $A_{\mathcal B}(i)=A_{\mathcal B}({j''})$. By the same argument as before, we have that $J'=\{i,j''\}\in {\mathcal H}\setminus {\mathcal N}$ and thus $\kappa_{\mathcal B}(\beta')=0$, where $\beta'=(\{i\},\{j''\})$. On the other hand, the path $(K_1,K_2,\ldots,K_u)$ belongs to ${\mathcal W}_{j''j}$. By induction hypothesis, there is a path $\beta''$ joining $j''$ and $j$ such that $\kappa_{\mathcal B}(\beta'')=0$. We end by taking $\beta=\beta'*\beta''$.
\end{proof}

\begin{proposition}
\label{prop:conexiondeA}
 For any $\{j\}\in A_{\mathcal B}(i)$, there is a $1$-path $\gamma$ in ${\mathcal H}$ joining $\{i\}$ and $\{j\}$ such that $\kappa_{\mathcal B}(\gamma)=0$.
\end{proposition}
\begin{proof} We know that there is a $1$-path $\sigma$ in $\mathcal H$ joining $\{i\}$ and $\{j\}$, such that $\kappa_{\mathcal B}(\sigma)$ is an even number.
Let us proceed by induction on $\kappa_{\mathcal B}(\sigma)$.
 If $\kappa_{\mathcal B}(\sigma)=0$, we are done. If $k_{\mathcal B}(\sigma)\geq 2$,  we can decompose
$
\sigma=\delta*\alpha*\rho
$,
where $\alpha$ is a $1$-path in ${\mathcal H}$ o of length $s+2\geq 2$ of the form
$$
\alpha: (\{i'\},\{i_1\},\{i_2\},\ldots,\{i_{s+1}\},\{i''\} ),
$$
satisfying that $K',K''\in {\mathcal B}$ and $T_\ell\notin{\mathcal B}$, for $\ell=1,2,\ldots,s$,  where $K'=\{i',i_1\}$, $K''=\{i_{s+1},i''\}$ and $T_\ell=\{i_\ell,i_{\ell+1}\}$.
Note that $\kappa_{\mathcal B}(\alpha)=2$ and hence $\{i''\}\in A_{\mathcal B}(i')$.
Now, it is enough to show that there is a $1$-path $\beta$ in $\mathcal H$ joining $\{i'\}$ and $\{i''\}$ such that $\kappa_{\mathcal B}(\beta)=0$. Indeed,
 the $1$-path
$\sigma_1= \delta*\beta*\rho
$
joins $\{i\}$ and $\{j\}$ and $k_{\mathcal B}(\sigma_1)=k_{\mathcal B}(\sigma)-2$.

Let us show the existence of $\beta$. If $i'=i''$, we take $\beta=(\{i'\})$ and we are done.
Assume that $i'\ne i''$. Note that $K\ne K'$, since if $K=K'$  the $1$-path $(\{i'\},\{i''\})$  contradicts the fact that $\{i''\}\in A_{\mathcal B}(i')$.
Since ${\mathcal B}$ is a $2$-connected component of $\mathcal N(2)$ in $\mathcal H$,  there is a $2$-path joining $K'$ and $K''$ with support in ${\mathcal B}$. Then ${\mathcal W}_{i'i''}$ is not empty and we apply Lemma \ref{lema:conexiondeA}.
\end{proof}

\begin{corollary}
\label{cor:unacomponente} For any index $i\in I$, the sets $A_{\mathcal B}(i)$ and $B_{\mathcal B}(i)$ are the $1$-connected components of
${\mathcal H}_{\mathcal B}={\mathcal H}\setminus \mbox{\rm Cl}_{\mathcal H}({\mathcal B})$.
\end{corollary}
\begin{proof} By Proposition \ref{prop:conexiondeA}  we deduce that $A_{\mathcal B}(i)$ is $1$-connected in ${\mathcal H}_{\mathcal B}$. Moreover, there is no $\{j\}\in B_{\mathcal B}(i)$ linked with an $\{i'\}\in A_{\mathcal B}(i)$ by a $1$-path in ${\mathcal H}_{\mathcal B}$, since that $1$-path should have at least one element in ${\mathcal B}$. Then $A_{\mathcal B}(i)$ is a $1$-connected component of ${\mathcal H}(1)$ in ${\mathcal H}_{{\mathcal B}}$. Recall that $B_{\mathcal B}(i)\ne\emptyset$. Taking $\{j\}\in B_{\mathcal B}(i)$, we know that $B_{\mathcal B}(i)=A_{\mathcal B}(j)$, hence $B_{\mathcal B}(i)$ is the second $1$-connected component of ${\mathcal H}_{\mathcal B}$.
\end{proof}
\begin{remark}
 \label{rk:unirayb}
If $\{j\}\in B_{\mathcal B}(i)$, there is a $1$-path $\gamma$ in ${\mathcal H}$ joining $\{i\}$ and $\{j\}$ such that $k_{\mathcal B}(\gamma)=1$. This is a direct consequence of Proposition \ref{prop:conexiondeA}. Take any $J=\{i_1,i_2\}\in {\mathcal B}$ and assume that $\{i_1\}\in A_{\mathcal B}(i)$, then $\{i_2\}\in B_{\mathcal B}(i)=A_{\mathcal B}$(j). Thus, we can join $\{i\}$ and $\{i_1\}$ with a $1$-path $\pi$ whith $k_{\mathcal B}(\pi)=0$ and we can also join $\{i_2\}$ and $\{j\}$ with a $1$-path $\rho$ whith $k_{\mathcal B}(\rho)=0$. Taking
$
\gamma=\pi*(\{i_1\},\{i_2\})*\rho,
$
we have that $k_{\mathcal B}(\gamma)=1$.
\end{remark}

\subsection{Number of Connected Components}  In this subsection, we complete the proof of Theorem \ref{teo:componentsconexas}. Let us fix an index $i_0\in I$ and an ordering
$
{\mathcal B}_1,{\mathcal B}_2,\ldots,{\mathcal B}_n,
$
in the list of nodal separating blocks of $\mathcal N$ in $\mathcal H$. Given $m\in \{0,1,\ldots,n\}$, we denote
$$
{\mathcal S}_m={\mathcal B}_1\cup{\mathcal B}_2\cup\cdots\cup{\mathcal B}_m,\quad
{\mathcal H}_m={\mathcal H}\setminus \operatorname{Cl}_{\mathcal H}({\mathcal S}_m).
$$
Note that ${\mathcal S}_n={\mathcal S}_{\mathcal H}^{\mathcal N}$ and ${\mathcal H}_n={\mathcal H}_{{\mathcal S}_{\mathcal H}}^{\mathcal N}$.
We also define
$
\Phi_m:{\mathcal H}(1)\rightarrow \{0,1\}^m
$
by
$$
\Phi_m(\{i\})(\ell)=\left\{
\begin{array}{ccc}
0&\mbox{ if }& \{i\}\in A_{{\mathcal B}_\ell}(i_0)\\
1&\mbox{ if }& \{i\}\in B_{{\mathcal B}_\ell}(i_0)
\end{array}
\right.
\quad, \mbox{ for } 1\leq \ell\leq m.
$$
Next Proposition \ref{prop:componentesconexas} implies Theorem \ref{teo:componentsconexas}:
\begin{proposition}
\label{prop:componentesconexas} For any $0\leq m\leq n$,  the image $\Delta_m$ of $\Phi_m$   has $m+1$ elements. Moreover, the $1$-connected components of ${\mathcal H}_m$ are the sets $\Phi_m^{-1}(c)$, where $c\in \Delta_m$.
\end{proposition}
Proposition \ref{prop:componentesconexas} comes from Lemma \ref{lema:dos}, Lemma \ref{lema:tres} and Corollary \ref{cor:tres} below:

\begin{lemma}
\label{lema:dos}
Consider an index $0\leq m<n$.
There is an element $\delta^{m}\in \Delta_m$ such that
$
 \{i_1\},\{i_2\} \in \Phi^{-1}_m(\delta^{m})$, for any  $J=\{i_1,i_2\}\in {\mathcal B}_{m+1}$.
\end{lemma}
\begin{proof} Take $J=\{i_1,i_2\}\in {\mathcal B}_{m+1}$ and select
$
\delta^{m}\in \Delta_m
$ such that $\{i_1\}\in \Phi_m^{-1}(\delta^{m})$.
If $\{i_2\}\notin \Phi_m^{-1}(\delta^{m})$, there is $1\leq \ell\leq m$ such that
$
\Phi_m(\{i_1\})(\ell)\ne \Phi_m(\{i_2\})(\ell).
$
That is, any $1$-path $\gamma$ in $\mathcal H$ joining $\{i_1\}$ and $\{i_2\}$ crosses ${\mathcal B}_\ell$ at an
odd number of entries of $\operatorname{Sub}(\gamma)$. This is not true, since we can take the path
$
(\{i_1\},\{i_2\})
$
where the only entry of the subsupport is $J$, but $J\in {\mathcal B}_{m+1}$ and ${\mathcal B}_\ell\cap {\mathcal B}_{m+1}=\emptyset$. Then, we have that $\{i_2\}\in \Phi_m^{-1}(\delta^{m})$.

Consider another $J'=\{i'_1,i'_2\}\in {\mathcal B}_{m+1}$. We can join $J$ with $J'$ by a $2$-path $\sigma$ in $\mathcal H$ with support in ${\mathcal B}_{m+1}$. Let us do induction on the length $\ell(\sigma)$ of $\sigma$. If $\ell(\sigma)=1$, we have $\sigma:(J,J')$, whith $K=J\cup J'\in {\mathcal H}(3)$. Up to reordering, we have that $i'_1\in J$ and hence
$\{i'_1\}\in \Phi_m^{-1}(\delta^{m})$. By the same argument as before, we conclude that  $\{i'_2\}\in \Phi_m^{-1}(\delta^{m})$. If $\ell(\sigma)>1$, we have
$
\sigma=(J,J_1)*\sigma'$, $\ell(\sigma')=\ell(\sigma)-1$,
where $\sigma'$ joints $J_1$ and $J'$. We know that $J_1=\{j_1,j_2\}$ with $\{j_1\},\{j_2\}\in \Phi_m^{-1}(\delta^{m})$. By induction, we conclude that $\{i'_1\},\{i'_2\}\in \Phi_m^{-1}(\delta^{m})$.
\end{proof}
\begin{lemma}
 \label{lema:tres}
 Consider an index $0\leq m<n$ and take $c\in \Delta_m$, with $c\ne \delta^{m}$. There is $\lambda_m(c)\in \{0,1\}$ such that $\Phi_{m+1}(\{i\})=(c,\lambda_m(c))$, for any $\{i\}\in \Phi^{-1}_m(c)$.
\end{lemma}

\begin{proof} Assume by contradiction that the statement is not true. This means that there are
$\{i\}, \{i'\}\in \Phi_{m}^{-1}(c)$ such that $\{i'\}\in B_{{\mathcal B}_{m+1}}(i)$. Since $\{i\}, \{i'\}\in \Phi_{m}^{-1}(c)$, we have that $\{i'\}\in A_{{\mathcal B}_\ell}(i)$, for $1\leq \ell\leq m$. Consider and index $\ell$ with $1\leq \ell\leq m$ such that $c_\ell\ne \delta^{m}_\ell$ where
$$
 c=(c_1,c_2,\ldots,c_m), \; \delta^{m}=(\delta^{m}_1,\delta^{m}_2,\ldots,\delta^{m}_m).
$$
Choose a $1$-path $\gamma$ in $\mathcal H$ joining $\{i\}$ and $\{i'\}$ such that $\kappa_{{\mathcal B}_\ell}(\gamma)=0$ and  write it as
$$
\gamma=(\{i\}=\{j_0\},\{j_1\},\ldots,{j_s}=\{i'\}).
$$
We know that $J_{t-1,t}=\{j_{t-1},j_t\}\notin {\mathcal B}_\ell$ for any $t=1,2,\ldots,s$. In particular, we have that $\Phi_{m}(\{j_t\})(\ell)=c_\ell$,  for any $t=0,1,\ldots,s$. Moreover, there is an odd number of indices $t$ such that $J_{t-1,t}\in {\mathcal B}_{m+1}$. In particular, there is a first $t_0$ such that $J_{t_0-1,t_0}\in {\mathcal B}_{m+1}$. By Lemma \ref{lema:dos}, we know that
$
\delta^{m}_\ell=\Phi_{m}(\{j_{t_0}\})(\ell)
$. This is the desired contradiction, since $ c_\ell \ne \delta^{m}_\ell$.
\end{proof}

\begin{corollary}
\label{cor:tres}
The set $\Delta_m$ has exactly $m+1$ elements,  for any $0\leq m\leq n$.
\end{corollary}
\begin{proof} We proceed by induction on $m$. If $m=0$ we are done, since  $\mathcal H(1)\ne \emptyset$ and $\{0,1\}^{\emptyset}$ has only one element, hence  there is only one mapping $\Phi_0$ that is surjective. If $m=1$, the result is given by Corollary \ref{cor:unacomponente}. Take $1\leq m \leq n$.   To show that $\Delta_{m}$ has $m+1$ elements, it is enough to see that
$$
\Delta_{m}=\{(c,\lambda_{m-1}(c));\; c\in \Delta_{m-1}\setminus\{\delta^{m-1}\}\}\cup \{\delta^{m-1}\}\times\{0,1\}.
$$
By Lemma \ref{lema:tres} we know that any $(c,\lambda)\in \Delta_{m}$ with $c\in \Delta_{m-1}\setminus \{\delta^{m-1}\}$ is such that $\lambda=\lambda_{m-1}(c)$. It remains  to show that both
$
(\delta^{m-1},0)$ and
$
(\delta^{m-1},1)
$
belong to $\Delta_{m}$. For this, take any $J=\{i_1,i_2\}\in {\mathcal B}_{m}$.  Up to reordering the elements in $J$, we have that $\{i_1\}\in A_{{\mathcal B}_{m}}(i_0)$ and $\{i_2\}\in B_{{\mathcal B}_{m}}(i_0)$, thus
$$
\Phi_{m}(\{i_1\})=(\delta^{m-1},0),\quad  \Phi_{m}(\{i_2\})=(\delta^{m-1},1).
$$
This ends the proof.
\end{proof}
\begin{lemma}
\label{lema:componentesconexas} Consider $0\leq m\leq n$, an element
$c\in \Delta_m$ and $\{i\},\{j\}\in\Phi_m^{-1}(c)$. There is a $1$-path $\gamma$ in $\mathcal H$ joining $\{i\}$ and $\{j\}$ such that $\kappa_{{\mathcal B}_\ell}(\gamma)=0$, for any $\ell=1,2,\ldots,m$.
\end{lemma}
\begin{proof} Take a $1$-path $\sigma$ in ${\mathcal H}$ joining $\{i\}$ and $\{j\}$. We know that $
\kappa_{{\mathcal B}_\ell}(\sigma)
$ is an even number, for any $\ell=1,2,\ldots, m$. Let us put $k_m(\sigma)=\sum_{\ell=1}^m\kappa_{{\mathcal B}_\ell}(\sigma)$. If $k_m(\sigma)=0$, we are done. Assume that $k_m(\sigma)>0$, hence $k_m(\sigma)\geq 2$. To conclude the proof, it is enough to find a $1$-path $\sigma'$ in ${\mathcal H}$ joining $\{ i \}$  and $\{ j \}$ such that $k_m(\sigma')<k_m(\sigma)$.

 Select $1\leq \ell\leq m$ such that $\kappa_{{\mathcal B}_\ell}(\sigma)>0$ and write $\sigma$ as
$$
\sigma=\alpha*(\{j_1\},\{j'_1\})*\tau*(\{j_2\},\{j'_2\})*\omega,
$$
where $\{j_1,j'_1\},\{j_2,j'_2\}\in {\mathcal B}_\ell$ are the extremal elements of ${\mathcal B}_\ell$ appearing as entries in the subsupport of $\sigma$. In particular $\Phi_m(\{j_1\})(\ell)=\Phi_m(j'_2)(\ell)$. Using Lemma  \ref{lema:conexiondeA} as in Proposition \ref{prop:conexiondeA}, we can substitute the $1$-path
$$
(\{j_1\},\{j'_1\})*\tau*(\{j_2\},\{j'_2\})
$$
by a $1$-path $
\beta= (\{j_1\}=\{j_1''\},\{j''_2\},\ldots,\{j''_s\}=\{j'_2\})$,
where the $\{j''_{t-1},j''_t\}\notin {\mathcal N}$. Now, we take
$
\sigma'=\alpha*\beta*\omega
$.
\end{proof}
\begin{corollary}
\label{cor:componentesconexas} For any $0\leq m\leq n$, the $1$-connected components of ${\mathcal H}_{m}$ are the sets $\Phi_m^{-1}(c)$, where $c\in \Delta_m$.
\end{corollary}
\begin{proof}
By Lemma \ref{lema:componentesconexas}, each $\Phi_m^{-1}(c)$ is $1$-connected in ${\mathcal H}_m$ and it is nonempty, since $c\in \Delta_m$. It remains to see that  there is no $1$-path joining a stratum $\{i\}\in \Phi_m^{-1}(c)$ with $\{i'\}\in \Phi_m^{-1}(c')$ in ${\mathcal H}_m$, when $c\ne c'$. Taking an index $1\leq\ell\leq m$ such that $c_\ell\ne c'_\ell$, we have $\{i'\}\in B_{{\mathcal B}_\ell}(i)$, then, any $1$-path $\gamma$ in $\mathcal H$ joining $\{i\}$ and $\{i'\}$ has $\kappa_{{\mathcal B}_\ell}\geq 1$ and it does not define a $1$-path in ${\mathcal H}_\ell$ and ``a fortiori'' in ${\mathcal H}_m$.
\end{proof}

\section{Combinatorially Simply connected Ambient Spaces}
This section is devoted to prove the following statement:
\begin{theorem}
\label{teo:estabilidadsimplyconnected}
Let  $f:{\mathcal M}\rightarrow {\mathcal M}_0$ be a standard transformation, where we have that ${\mathcal M}_0=({\mathbb C}^d, E_0; \{0\})$. The combinatorial strata structure ${\mathcal H}_{\mathcal M}$ is simply connected.
\end{theorem}
 Theorem \ref{teo:estabilidadsimplyconnected} is a consequence of Proposition \ref{prop:estabilidadsimplyconnected} below.
\begin{proposition}
\label{prop:estabilidadsimplyconnected}
Let $\pi:{\mathcal M}'\rightarrow {\mathcal M}$ be a standard blow-up with center $Y$, where ${\mathcal H}_{\mathcal M}$ is simply connected. Then ${\mathcal H}_{{\mathcal M}'}$ is simply connected.
\end{proposition}

\subsection{The Topological Simplicial Complex} Before starting the proof of Proposition \ref{prop:estabilidadsimplyconnected}, let us consider the natural topological object $\Omega_{\mathcal H}$ associated to a given combinatorial strata structure $\mathcal H$ with minimal set of indices $I$.

For any $i\in I$, let us denote by $\xi^I_i\in {\mathbb R}^I$ the characteristic function of $\{i\}\subset I$. Given $J\subset I$, the {\em simplex $\Delta^I_J$} is the convex hull in ${\mathbb R}^I$ of the set
$\{\xi^I_i; i \in J\}$. By convention, we take $\Delta_\emptyset^I=\emptyset$. The faces $\Delta^I_{J'}$ of $\Delta^I_J$ correspond to the subsets $J'\subset J$.

Given a subset $A\subset {\mathcal P}(I)$, we denote by $\Omega^I_A$ the topological simplicial complex defined by
$
 \Omega^I_A=\cup_{J\in A}\Delta^I_J
$, considered as a topological subspace of ${\mathbb R}^I$. We see that $A$ is a combinatorial strata structure if and only if it has the maximality property:
$
\Omega^I_{A'}=\Omega^I_A\Rightarrow A'\subset A
$.

If $I\subset I'$, we have a closed immersion ${\mathbb R}^I\subset {\mathbb R}^{I'}$ that allows us to identify  $\Omega^I_A$ with $\Omega^{I'}_A$;
we denote just $\Omega_A=\Omega_A^I=\Omega_A^{I'}$.

Consider an element $\infty\notin I$ and a subset  $A\subset {\mathcal P}(I)$.  We denote
$$
A*\{\infty\}= \{J\cup \{\infty\};\; J\in A\}.
$$
The topological simplicial complex $\Omega_{A*\{\infty\}}\subset {\mathbb R}^{I\cup\{\infty\}}$ is the {\em cone } of $\Omega_{A}$ over $\infty$. We denote ${\mathcal C}\Omega_{A}=\Omega_{{A}*\{\infty\}}$, when no confusion arises.

\begin{proposition}
 \label{prop:simplyconnected}
 Let ${\mathcal H}\subset {\mathcal P}(I)$ be a combinatorial strata structure. Then
 $\mathcal H$ is $1$-connected if and only if $\Omega_{\mathcal H}$ is connected. Moreover
 $\mathcal H$ is simply connected if and only if $\Omega_{{\mathcal H}_3}$ is a simply connected topological space.
\end{proposition}
\begin{proof} See  Appendix \ref{Appendix: Simply Connectedness}.
\end{proof}

\subsection{Combinatorial Strata Sructure After Blow-up} We start here the proof of Proposition \ref{prop:estabilidadsimplyconnected}.
In order to simplify notations, let us denote ${\mathcal H}={\mathcal H}_{\mathcal M}$ and ${\mathcal H}'={\mathcal H}_{{\mathcal M}'}$. We know that ${\mathcal H}_3$ is simply connected and we have to prove that ${\mathcal H}'_3$ is simply connected, see Remark
\ref{rk:tresdimensionalcase}. In view of Proposition  \ref{prop:simplyconnected}, it is enough to see that the topological simplicial complex $\Omega_{{\mathcal H}'_3}$ is simply connected.

We start by a description of ${\mathcal H}'$ and of ${\mathcal H}'_3$. Recall that $E=\cup _{i\in I}E_i$. Let us take a symbol $\infty \notin I$. We have that
$
E'=E'_\infty\cup ( \cup_{i\in I}E'_i)
$,
where the $E'_i$ are the strict transforms of $E_i$, for $i\in I$, and $E'_\infty=\pi^{-1}(Y)$ is the exceptional divisor of the blowing-up. Hence, the minimal set of indices for ${\mathcal H}'$ is $I'=I\cup \{\infty\}$. Define $\widetilde{\mathcal Z}_Y$ and $\widetilde{\mathcal B}_Y$ by
$$
\widetilde{\mathcal Z}_Y=\{J\in {\mathcal H};\quad E_J\subset Y \},
\quad
\widetilde{\mathcal B}_Y= \{J\in {\mathcal H};\quad E_J\cap Y\ne \emptyset,\;  E_J\not\subset Y \}.
$$
In view of the usual properties of the blow-ups, we describe ${\mathcal H}'$ by the following equalities:
$$
{\mathcal H}'\cap {\mathcal P}(I)={\mathcal H}\setminus \widetilde{\mathcal Z}_Y,\quad
{\mathcal H}'\setminus {\mathcal P}(I)=\widetilde{\mathcal B}_Y*\{\infty\}.
$$
Note that $\widetilde{\mathcal Z}_Y$ defines the strata disappearing after the blowing-up and $\widetilde{\mathcal B}_Y$ the strata whose strict transform intersects the exceptional divisor. Let us denote
$
{\mathcal Z}_Y={\mathcal H}_3\cap\widetilde{\mathcal Z}_Y
$
 and
 $
{\mathcal B}_Y={\mathcal H}_2\cap\widetilde{\mathcal B}_Y
$, then we have
\begin{equation}
\label{eq:hprimatres}
{\mathcal H}'_3\cap {\mathcal P}(I)={\mathcal H}_3\setminus {\mathcal Z}_Y; \quad
{\mathcal H}'_3\setminus {\mathcal P}(I)={\mathcal B}_Y*\{\infty\}.
\end{equation}
Let us note that ${\mathcal Z}_Y$ is closed in ${\mathcal H}_3$ and ${\mathcal B}_Y$ is open in ${\mathcal H}_2$. Hence both ${\mathcal H}'_3\cap {\mathcal P}(I)$ and ${\mathcal B}_Y$ are combinatorial strata structures.
\begin{lemma}
\label{lema:centersofblowup}
One of the following possibilities holds:
\begin{enumerate}
\item[a)] ${\mathcal Z}_Y=\emptyset$.
\item[b)] There is $J_0\in {\mathcal H}(3)$ such that  $Y=E_{J_0}$.
\item[c)] There is $J_0\in {\mathcal H}(2)$ such that  $Y=E_{J_0}$.
\end{enumerate}
\end{lemma}
\begin{proof} Assume that ${\mathcal Z}_Y\ne\emptyset$. This means that there is $J_0\in {\mathcal H}_3$ such that $E_{J_0}\subset Y$. Note that $J_0\in {\mathcal H}(2)\cup {\mathcal H}(3)$, since the center $Y$ has codimension $\geq 2$.
If we can take $J_0\in {\mathcal H}(2)$, by codimensional reasons we have that  $Y=E_{J_0}$ and c) holds.  Assume that there is no  $J\in {\mathcal H}(2)$ with $E_J\subset Y$. Then $J_0\in{\mathcal H}(3)$; let us put $J_0=\{i_1,i_2,i_3\}$. If $Y$ has codimension $3$, we have $Y=E_{J_0}$ and b) holds. Suppose that $Y$ has codimension $2$ and take a point $p$ in the stratum $S_{J_0}\subset E_{J_0}$. There are local coordinates
$$(z_j)_{j=1}^d=(x_{1},x_{2},x_{3},y_4,\ldots,y_d)$$
at $p$ such that $E_{i_j}=(x_j=0)$, $j=1,2,3$ and $Y=(z_a=z_b=0)$, with the property that
$
(x_1=x_2=x_3=0)\subset (z_a=z_b=0)
$.
Then $a,b\in \{1,2,3\}$ and b) holds.
\end{proof}
\begin{remark}
\label{rk:zset}
 In case b) of Lemma
\ref{lema:centersofblowup}, we have that ${\mathcal Z}_{Y}=\{J_0\}$. In case c), we have that
$
{\mathcal Z}_{Y}=\{J_0\}\cup \{J\in {\mathcal H}(3); \; J\supset J_0\}
$ and ${\mathcal B}_Y=\{J\in {\mathcal H}_2;\; J\ne J_0, J\cup J_0\in {\mathcal H}\}$.
\end{remark}
\subsection{Stability of Simple Connectedness}
We end the proof of Proposition \ref{prop:estabilidadsimplyconnected}, looking at the cases a), b) and c) in Lemma
\ref{lema:centersofblowup}.
Before starting the case by case study, let us consider ${\mathcal A}_Y\subset {\mathcal H}$ defined by:
$$
{\mathcal A}_Y=\{J\in {\mathcal H}_2; E_J\cap Y\ne\emptyset\}.
$$
It is an open set of ${\mathcal H}_2$ and hence it is a combinatorial strata structure. Note that ${\mathcal B}_Y\subset {\mathcal A}_Y$. More precisely, in cases a) and b) of Lemma \ref{lema:centersofblowup}, we have that ${\mathcal B}_Y={\mathcal A}_Y$; in case c), we have that
${\mathcal B}_Y\cup \{J_0\}={\mathcal A}_Y$, where $Y=E_{J_0}$, with $J_0\in {\mathcal H}(2)$.

\begin{lemma}
\label{lema:connexionfrontera}
 ${\mathcal A}_Y$ is $1$-connected.
\end{lemma}
\begin{proof} Let $E_i$ and $E_j$ be two irreducible components of $E$ such that $Y\cap E_i\ne\emptyset$ and $E_j\cap Y\ne\emptyset$. That is, we have  $\{i\},\{j\}\in {\mathcal A}_Y(1)$. We need to show that there is a $1$-path
$$
\gamma= (\{i\}=\{i_1\},\{i_2\},\ldots,\{i_{t+1}\}=\{j\})
$$
such that $\{i_s,i_{s+1}\}\in {\mathcal A}_Y(2)$, for $s=1,2,\ldots,t$. We know that $Y\cap E$ is connected. Thus, joining a point in $E_i\cap Y$ with a point in $E_j\cap Y$ by a topological path contained in $E$, we have a sequence
$$
\{i\}=\{i^0_1\}=J_0, J_1,\ldots, J_{k}\supset \{j\},
$$
such that $J_s\in{\mathcal H}$, with $Y\cap E_{J_s}\ne\emptyset$ and  $J_{s-1} \cap J_{s} \ne \emptyset$, for $s=1,2,\ldots, k$. Let us do induction on the length $k$, in order to obtain $\gamma$. If $k=1$, we have
$$
\{i\}=J_0, J_1\supset\{j\}.
$$
Note that $i\in J_1$, since $J_0\cap J_1\ne\emptyset$.
Then $E_{\{i,j\}}\supset E_{J_1}$ and hence $Y\cap E_{\{i,j\}}\ne\emptyset$, that is $\{i,j\}\in {\mathcal A}_Y$. We can take $\gamma=(\{i\},\{j\})$.
 Assume that $k\geq 2$. If $i\in J_2$, we are done, since we find a length $k-1$ sequence
$
\{i\}, J_2,J_3,\ldots, J_{k}\supset\{j\}
$
and we apply the induction hypothesis. Assume that $i\notin J_2$. Take $i'\in J_1\cap J_2$, we see that $\{i,i'\}\subset J_1$ and hence $\{i,i'\}\in {\mathcal A}_Y$. We end by taking $\gamma=(\{i\},\{i'\})*\gamma'$, where $\gamma'$ is given by applying induction to the sequence
$
\{i'\}, J_2,J_3,\ldots, J_{k}\supset\{j\}.
$
\end{proof}

\subsubsection  {The case of empty ${\mathcal Z}_Y$} We know that ${\mathcal B}_Y= {\mathcal A}_Y$. Hence, we have
$
{\mathcal H}'_3\cap {\mathcal P}(I)={\mathcal H}_3$ and
$
{\mathcal H}'_3\setminus {\mathcal P}(I)={\mathcal A}_Y*\{\infty\}
$.
Then
$
{\mathcal H}'_3={\mathcal H}_3\cup ({\mathcal A}_Y*\{\infty\})
$ and thus $\Omega_{{\mathcal H}'_3}$ is given by
$$
\Omega_{{\mathcal H}'_3}=\Omega_{{\mathcal H}_3}
\cup \mathcal{C}{\Omega_{{\mathcal A}_Y}};\quad  \Omega_{{\mathcal H}_3}\cap
\mathcal{C}{\Omega_{{\mathcal A}_Y}}={\Omega_{{\mathcal A}_Y}}.
$$
By
Lemma \ref{lema:connexionfrontera}, the topological simplicial complex  $\Omega_{{\mathcal A}_Y}$ is connected and hence the cone  $\mathcal{C}{\Omega_{{\mathcal A}_Y}}$ is simply connected. Since $\mathcal{C}{\Omega_{{\mathcal A}_Y}}$ and $\Omega_{{\mathcal H}_3}$ are simply connected and ${\Omega_{{\mathcal A}_Y}}$ is connected, by an application of a combinatorial version of  Seifert-van Kampen's Theorem, see \cite{Koz}, we conclude that $\Omega_{{\mathcal H}'_3}$ is simply connected.

\subsubsection{The case $Y=E_{J_0}$, with $J_0=\{i_1,i_2,i_3\}\in {\mathcal H}(3)$} We have that ${\mathcal B}_Y={\mathcal A}_Y$ and ${\mathcal Z}_{Y}=\{J_0\}$.   Moreover
$$
{\mathcal A}_Y=\{J\in {\mathcal H}_2;\; E_J\cap Y\ne\emptyset\}=\{J\in {\mathcal H}_2;\; J\cup J_0\in {\mathcal H}\}.
$$
Then, we have
$
\Omega_{{\mathcal H}'_3}=\Omega_{{\mathcal H}_3\setminus \{J_0\}}
\cup C{\Omega_{{\mathcal A}_Y}}
$
and
 $
 \Omega_{{\mathcal H}_3\setminus \{J_0\}}\cap
C{\Omega_{{\mathcal A}_Y}}={\Omega_{{\mathcal A}_Y}}
$.

Let us note that $\Omega_{{\mathcal H}_3\setminus\{J_0\}}$ is obtained from $\Omega_{{\mathcal H}_3}$ by removing the $2$-dimensional simplex $\Delta_{J_0}$. Since  $\Omega_{{\mathcal H}_3}$ is simply connected, the fundamental group of $\Omega_{{\mathcal H}_3\setminus\{J_0\}}$ is generated by the single loop
$
(\{i_1\},\{i_2\},\{i_3\},\{i_1\})
$.
This loop is defined in ${\mathcal A}_Y$ and hence it is homotopically trivial in the cone ${\mathcal C}{\mathcal A}_Y$. Noting that $\Omega_{{\mathcal A}_Y}$ is connected and Applying Seifert-van Kampen's theorem, we get that $\Omega_{{\mathcal H}'_3}$ is simply connected.
\subsubsection{The case $Y=E_{J_0}$, with $J_0=\{i_1,i_2\}\in {\mathcal H}(2)$.}  Let $L\subset I$ be the set of $k\in I$ such that  $J_0\cup\{k\}\in {\mathcal H}(3)$. By Remark \ref{rk:zset}, we have
$$
{\mathcal Z}_Y=\{J_0\}\cup \{J_0\cup\{k\};\; k\in L\},\quad
{\mathcal B}_Y=\{J\in {\mathcal H}_2;\; J\ne J_0, J\cup J_0\in {\mathcal H}\}.
$$
We know that
$
\Omega_{{\mathcal H}'_3}
=
 \Omega_{
({\mathcal H}_3\setminus {\mathcal Z}_Y)
}
\cup \mathcal{C}
\Omega_{
{\mathcal B}_Y
}
$.

Let us consider the cases $L=\emptyset$ and $L\ne\emptyset$. If $L=\emptyset$, we have that
${\mathcal Z}_Y=\{J_0\}$ and ${\mathcal B}_Y=\{\{\emptyset\},\{i_1\},\{i_2\}\}$. In this case, we have that $\Omega_{{\mathcal H}_3\setminus {\mathcal Z}_Y}$ is obtained from $\Omega_{{\mathcal H}_3}$ by removing the simplex $\Delta_{i_1,i_2}$; moreover, we add the union ${\mathcal C}\Omega_{{\mathcal B}_Y}=\Delta_{i_1,\infty}\cup\Delta_{\infty,i_2}$ to
$\Omega_{{\mathcal H}_3\setminus {\mathcal Z}_Y}$  in order to
get $\Omega_{{\mathcal H}'_3}$. Identifying $\Delta_{i_1,i_2}$ and $\Delta_{i_1,\infty}\cup\Delta_{\infty,i_2}$ in a convenient way, we obtain a homeomorphism between $\Omega_{{\mathcal H}_3}$  and $\Omega_{{\mathcal H}'_3}$.

Let us assume now that $L\ne\emptyset$. We have that ${\mathcal B}_Y={\mathcal B}'_Y\cup {\mathcal B}''_Y$, where
\begin{eqnarray*}
{\mathcal B}'_Y&=&\{\{i_1\},\{i_2\}\}\cup \{\{i_1,k\};\; k\in L\}
\cup \{\{i_2,k\};\; k\in L\}.
\\
{\mathcal B}''_Y&=&\{\emptyset\}\cup \{\{k\}; \; k\in L\}\cup
\{
\{k,k'\}; k,k'\in L,\; \{k,k'\}\cup J_0\in {\mathcal H}(4)
\}.
\end{eqnarray*}
In this situation ${\mathcal B}_Y$ is $1$-connected. Indeed, we have
$${\mathcal B}_Y(1)=\{\{i_1\},\{i_2\}\}\cup \{\{k\};k\in L\}.$$
Taking $k_0\in L$, we see that $\{i_1\}$ and $\{i_2\}$ are connected through $\{k_0\}$ and any other $\{k\}$, with $k\in L$, is connected with $\{k_0\}$ through $\{i_1\}$ and $\{i_2\}$. Now, we have
$$
\Omega_{{\mathcal H}'_3}={\Omega}'\cup \mathcal{C}\Omega_{{\mathcal B}''_Y},
\quad
\widetilde{\Omega}\cap \mathcal{C}\Omega_{{\mathcal B}''_Y}=\Omega_{{\mathcal B}_Y},
$$
where $
{\Omega}'=
\Omega_{
({\mathcal H}_3\setminus {\mathcal Z}_Y)
}
\cup
\mathcal{C}\Omega_{
{\mathcal B}'_Y
}
$.

Let us see that $\Omega_{{\mathcal H}_3}$ and ${\Omega}'$ are homeomorphic topological spaces. In particular ${\Omega}'$ is simply connected. Let us recall that
$$
{\Omega}'=\Omega_{{\mathcal H}_3\setminus{\mathcal Z}_Y}\cup \Delta_{i_1,\infty}\cup \Delta_{i_2,\infty}\cup \bigcup_{k\in L}(\Delta_{i_1,k,\infty}\cup \Delta_{i_2,k,\infty}).
$$
Now, it is enough to give the homeomorphism (in a compatible way with the simplicial structure) over the simplices $\Delta_{J}$, where $J\in {\mathcal Z}_Y$. Let us take $J\in {\mathcal Z}_Y$. If $J=J_0$, we identify
$
\Delta_{J_0}=\Delta_{i_1,\infty}\cup \Delta_{i_2,\infty}
$.
If $J\ne J_0$, then $J=\{i_1,i_2, k\}\in {\mathcal H}(3)$, with $k\in L$, and we identify
$
\Delta_{J}=\Delta_{i_1,k,\infty}\cup \Delta_{i_2,k,\infty}
$.

Now, in order to see that
$
\Omega_{{\mathcal H}'_3}=
{\Omega}'\cup\mathcal{C}\Omega_{{\mathcal B}''_Y}
$
is simply connected, it is enough to apply Seifert-Van Kampen's Theorem, since $\Omega'$ and ${\mathcal C}\Omega_{{\mathcal B}''_Y}$  are simply connected and their intersection $\Omega_{{\mathcal B}_Y}$ is connected.

\section{Invariant Hypersurfaces and Separator Set}
In this section we give a proof of Theorem \ref{teo:main}. Let us recall the statement:
\begin{quote}\em
Let
$
\pi:
(\mathcal M,{\mathcal F})\rightarrow ({\mathcal M}_0,{\mathcal F}_0)
$
be a nodal reduction of singularities
of a
GH-foliated space $({\mathcal M}_0,{\mathcal F}_0)$, where
$$
{\mathcal M}_0=(({\mathbb C}^n,\emptyset;\{0\}),\quad \mathcal M=(M,E;K).
$$
Consider a connected component $C$ of $E\setminus \vert{\mathcal S}_{\mathcal M}^{\mathcal F}\vert$, where $\mathcal S_{\mathcal M}^{\mathcal F}$ denotes the nodal separator set.
Then, there is an invariant hypersurface $H_0$ of ${\mathcal F}_0$ whose strict transform $H$ satisfies that $H\cap E\subset C$.
\end{quote}
Let us note that the topological closure $\overline C$ of $C$ is the union of the irreducible components of $E$ that intersect $C$. Now,  Theorem \ref{teo:main} is a direct consequence of the following technical, but more general statement:
\begin{theorem}
\label{teo:maintechnical}
 Let $({\mathcal M}_0,{\mathcal F}_0)$ be a GH-foliated space over
 ${\mathcal M}_0=({\mathbb C}^n,E^0;\{0\})$. Consider a nodal reduction of singularities
$
\pi:({\mathcal M},{\mathcal F})\rightarrow ({\mathcal M}_0,{\mathcal F}_0)
$.
Given a  connected component $C$ of $E\setminus \vert{\mathcal S}_{\mathcal M}^{\mathcal F}\vert$, one of the following properties holds:
 \begin{enumerate}
 \item[a)] There is an irreducible component $E^0_k$ of $E^0$ such that the strict transform $E_k$ of $E^0_k$  is contained in $\overline{C}$.
 \item[b)] There is an invariant branch $(\Gamma_0,0)$ of ${\mathcal F}_0$ not contained in $E^0\cup \operatorname{Sing}({\mathcal F}^0, E^{0})$,
 such that the strict transform $(\Gamma,p)$ of $(\Gamma_0,0)$  meets $\overline C$.
 \end{enumerate}
\end{theorem}
Let us see how to obtain Theorem \ref{teo:main} from Theorem \ref{teo:maintechnical}.
We have that $E^0=\emptyset$. Hence there is an invariant branch $(\Gamma_0,0)$ as stated in property b) of  Theorem \ref{teo:maintechnical}. Since the centers of the blow-ups are contained in the adapted singular locus, we have that $(\Gamma,p)\not\subset E$. Hence $\Gamma\cap \overline C=\{p\}$. The point $p$ is necessarily a singular point of $\mathcal F$, but it cannot be of corner type, since there is no invariant branches arriving to simple corners outside $E$. Hence $p$ is of trace type, in particular $p\notin \vert{\mathcal S}_{\mathcal M}^{\mathcal F}\vert$, that is $p\in C$ and it belongs to a trace type component of the singular locus, that defines a partial separatrix $\Sigma$. The closed hypersurface $(H_\Sigma, H_\Sigma \cap K)$ of $(M,K)$ (recall that $K=\pi^{-1}(0)$) projects by $\pi$ onto the desired $H_0$.

\subsection{Structure of the proof}
The proof of Theorem \ref{teo:maintechnical} goes by induction on the length of $\pi$ as a composition of a finite sequence of admissible blow-ups, besides to a reduction to the two dimensional case.

\begin{remark} The connected component $C$ determines a set of indices $I_C\subset I$ such that $E_i\subset \overline C$ if and only if $i\in I_C$. We use the intuitive terminology ``$E_i$ belongs to $C$'' for denoting the property that $i\in I_C$.
\end{remark}

If the length of $\pi$ is zero, then $\pi$ is the identity morphism and the origin is a simple point of $({\mathcal M}_0,{\mathcal F}_0)$. If $p\in \operatorname{Sing}({\mathcal F})$, there is always an irreducible component of $E^0$ belonging to $C$ and we are done. If $E^0=\emptyset$ and $p$ is a regular point, then $C$ does not exist, and the statement is tautologically true.

Assume that the length of $\pi$ is $\geq 1$. We decompose $\pi$ as $\pi=\pi_1\circ\sigma$, where
$$
({\mathcal M},{\mathcal F})\stackrel{\sigma}{\rightarrow} ({\mathcal M}_1,{\mathcal F}_1)\stackrel{\pi_1}{\rightarrow} ({\mathcal M}_0,{\mathcal F}_0)
$$
and $\pi_1$ is an admissible blow-up with center $Y_0$. Let us assume that none of the strict transforms of the irreducible components of $E^0$ belong to $C$. Denote by $E^1_\infty=\pi_1^{-1}(Y_0)$ the exceptional divisor of the first blow-up $\pi_1$. We have two cases to consider
\begin{enumerate}
\item[A)] The strict transform $E_\infty$ of $E^1_\infty$ in ${\mathcal M}$ belongs to $C$.
\item[B)] $E_\infty$ does not belong to $C$.
\end{enumerate}

 If we are in case B), we apply induction on the length of a reduction of singularities at a point $q\in \sigma(C)$.
 It remains to consider the case A). Before doing it, we need to introduce the definition of {\em two-equirreduction points}, see also \cite{Can-M}.

 \subsection{Two-Equireduction} We consider small enough representatives of the germs in our arguments, keeping the same notations, in the hope that the reader would give the necessary precisions without too much difficulty. Let us consider the sequences of blowing-ups $\pi=\pi_1\circ \sigma$ as before.
Recall that ${\mathcal M}_0=(M_0,E^0;\{0\})$, where $M_0$ is a  representative of the germ $(M_0,\{0\})$.

Take a point $p\in \operatorname{Sing}({\mathcal F}_0,E^0)$. We say that $p$ is a {\em two-equireduction point of height $0$ for $({\mathcal M}_0,{\mathcal F}_0)$ with respect to $\pi$} if and only if $p$ is a simple singular point for $({\mathcal M}_0,{\mathcal F}_0)$ of dimensional type two and the morphism $\pi$ is an isomorphism in a neighborhood of $p$ (there are no blow-ups over $p$). Denote $Y=\operatorname{Sing}({\mathcal F}_0,E^0)$ the adapted singular locus;  in this case we have the following properties, locally at the point $p$:
\begin{enumerate}
\item The adapted singular locus $Y$ is a non singular codimension two subspace $(Y,p)\subset (M_0,p)$ and it has normal crossings with $E^0$.
\item $\operatorname{Sing}({\mathcal F}_0)=Y$.
\item Each irreducible component of $E^0$ through $p$ contains $Y$.
\end{enumerate}
Given an integer number $k\geq 1$, we say that $p$ is a {\em two-equireduction point of height $k$ for $({\mathcal M}_0,{\mathcal F}_0)$ with respect to $\pi$} if the following properties hold:
\begin{enumerate}
\item  The germ $(Y_0,p)$ of the adapted singular locus $\operatorname{Sing}({\mathcal F}_0,E^0)$ is non singular of codimension two and $Y_0$ is the center of $\pi_1$.
\item The blow-up $\pi_1$ induces and \'etale map of germs
$
(Y_1,\pi_1^{-1}(p)\cap Y_1) \rightarrow (Y_0,p),
$
 where
$Y_1= \operatorname{Sing}({\mathcal F}_1,E^1)$.
\item Any $p_1$ in $\operatorname{Sing}({\mathcal F}_1,E^1)\cap \pi^{-1}(p)$ is  a {two-equirreduction point of height lower or equal than $k-1$ for $({\mathcal M}_1,{\mathcal F}_1)$, with respect to $\sigma$} and one of such points has height equal to $k-1$.
\end{enumerate}
We say that $p$ is an {\em two-equirreduction point for $({\mathcal M}_0,{\mathcal F}_0)$ with respect to $\pi$} if it is a two-equirreduction point of height $k$, for some
$k\geq 0$. Let us denote by $\operatorname{Equir}_{\pi}({\mathcal M}_0,{\mathcal F}^0)$ the set of two-equirreduction points.

The proof of next Propositions \ref{pro:codimensionnonequireduction} and
\ref{pro:transveresalsection}
 is essentially contained in \cite{Can-M}:
\begin{proposition}
\label{pro:codimensionnonequireduction}
The set
$
Z_0= \operatorname{Sing}({\mathcal F}_0,E^0)\setminus \operatorname{Equir}_{\pi}({\mathcal M}_0,{\mathcal F}^0)
$
is a closed analytic set of codimension bigger or equal than three in $M_0$.
\end{proposition}
Let us consider a two dimensional nonsingular subspace $(\Delta,p)\subset (M,p)$. We recall that $(\Delta,p)$ is a {\em strict transversal to ${\mathcal F}$} if ${\mathcal F}\vert\Delta$ is a foliation generated by $\omega\vert_{\Delta}$, where $\omega$ is a local holomorphic generator of ${\mathcal F}$ at $p$. That is, the coefficients of $\omega$ have no common factor and the coefficients of $\omega\vert_{\Delta}$  are also without common factor (see \cite{Mat-Mou} for the existence of such sections).
\begin{proposition}
\label{pro:transveresalsection}
Let
$(\Delta_0,p)\subset (M,p)$ be a two dimensional non singular subspace transversal to  $\operatorname{Sing}({\mathcal F}_0,E^0)$ at a point $p\in \operatorname{Equir}_{\pi}({\mathcal M}_0,{\mathcal F}^0)$. Then $(\Delta_0,p)$ is a strict transversal to $\mathcal F$  and $\pi$ induces a
two dimensional reduction of singularities
$$
\pi\vert_{\Delta}: (\Delta, \Delta\cap E, {\mathcal F}\vert_{\Delta})\rightarrow
(\Delta_0, \Delta_0\cap E^0, {\mathcal F}_0\vert_{\Delta_0}),
$$
where $\Delta=\pi^{-1}(\Delta_0)$.
\end{proposition}
\begin{proof} See \cite{Can-M}. Let us remark that $\Delta$ is equal to the strict transform of $\Delta_0$.
	\end{proof}
 \subsection{Reduction to dimension two} Let us choose $\epsilon_0>0$ such that, for any $\epsilon$ with $0<\epsilon\leq\epsilon_0$, the sequence of blow-ups $\pi$ is ``well represented'' by a  nodal reduction of singularities
 $\pi^\epsilon:({\mathcal M}^\epsilon,{\mathcal F}^\epsilon)\rightarrow ({\mathcal M}_0^\epsilon,{\mathcal F}_0^\epsilon)
 $,
such that
 \begin{equation*}
 \begin{array}{lccc}
 {\mathcal M}_0^\epsilon=(M_0^\epsilon,E^0(\epsilon)),&
 M_0^\epsilon=\operatorname{B}(\mathbf{0};\epsilon),&
 E^0(\epsilon)=E^0\cap M_0^\epsilon,&
{\mathcal F}_0^\epsilon={\mathcal F}_0\vert_{M^\epsilon_0};\\
{\mathcal M}^\epsilon=(M^\epsilon,E(\epsilon)),&
 M^\epsilon=\pi^{-1}(\operatorname{B}(\mathbf{0};\epsilon)),&
 E(\epsilon)=E\cap M^\epsilon,&
{\mathcal F}^\epsilon={\mathcal F}\vert_{M^\epsilon};
 \end{array}
 \end{equation*}
where $\operatorname{B}(\mathbf{0};\epsilon)\subset{\mathbb C}^n$ is the ball
centered at the origin of radius $\epsilon$.
  In particular, the following properties hold:
 	\begin{enumerate}
 		\item The germ along $\pi^{-1}(0)$ of the center of each blow-up in $\pi^\epsilon$ coincides with the corresponding center in the sequence $\pi$. Hence, the germ along
 $\pi^{-1}(0)$ of each stratum defined by the divisor $E(\epsilon)\subset  M^
 \epsilon$ coincides with the corresponding stratum for $E$.
 		\item The germ along $\pi^{-1}(0)$ of $\operatorname{Sing}({\mathcal F}^\epsilon)$ coincides with $\operatorname{Sing}({\mathcal F})$. Moreover, the irreducible components of  $\operatorname{Sing}({\mathcal F}^\epsilon)$ coincide one by one with the ones for $\operatorname{Sing}({\mathcal F})$.
 		\item Each partial separatrix of $({\mathcal M}^\epsilon,{\mathcal F}^\epsilon)$ defines a unique irreducible invariant hypersurface in $M^\epsilon$, whose germ  along $\pi^{-1}(0)$ coincides with the corresponding invariant hypersurface for $\mathcal F$.
 	\end{enumerate}
 Note that we can split $\pi^\epsilon$ as $\pi^\epsilon=\pi_1^\epsilon\circ\sigma^\epsilon$.
 We know that the set $Z_1^\epsilon$ given by
 $$
 Z_1^\epsilon=
 \operatorname{Sing}({\mathcal F}_1^\epsilon,E^1(\epsilon))\setminus \operatorname{Equir}_{\sigma^\epsilon}({\mathcal M}_1^\epsilon,{\mathcal F}_1^\epsilon)
 $$
 is a closed analtyic set of $M_1^\epsilon$ with codimension greater or equal than three.

 Let $Y_0^\epsilon\subset M_0^\epsilon$ be the center of the blow-up $\pi_1^\epsilon$ and consider the exceptional divisor $E_\infty^1(\epsilon)\subset E^1(\epsilon)$ of $\pi_1^{\epsilon}$. We have an induced fibration
 $$
 \pi_1^\epsilon: E_\infty^1(\epsilon)\rightarrow Y_0(\epsilon),
 $$
 such that the fiber is isomorphic to ${\mathbb P}^{n-d-1}_{\mathbb C}$, where $d=\dim Y_0$.
The intersection $Z_1^\epsilon\cap E_\infty^1(\epsilon)$ has codimension $\geq 2$ in $ E_\infty^1(\epsilon)$. Then, there is an open dense subset $U\subset Y_0^\epsilon$ such that the codimension of
 $Z_1^\epsilon\cap\pi_1^{-1}(p)$ in $\pi^{-1}(p)$ is greater or equal than two, for any $p\in U$ (see \cite{Gri-H}). Let us decompose
 $\operatorname{Sing}({\mathcal F}^\epsilon_1,E^1(\epsilon))$ as union of two closed analytic sets
 $
 S'\cup S''
 $,
 where $S'$ is the union of the irreducible components of  $\operatorname{Sing}({\mathcal F}^\epsilon_1,E^1(\epsilon))$  contained in $E^1_\infty$ and $S''$ the union of the other ones. Hence we have
 \begin{enumerate}
 	\item Both $S'$ and $S''$ are of codimension $\geq 2$ in $M_1^\epsilon$.
 	\item The codimension of $S'$ in $E^1_\infty(\epsilon)$ is $\geq 1$.
 	\item $S'\cap S''\subset S''\cap E^1_\infty(\epsilon)\subset Z^\epsilon_1$.
 	\item The singular locus of $S'$ is contained in $Z^\epsilon_1$.
 \end{enumerate}
 	Now, let us consider the closed analytic subset set $T_1\subset S'$, defined by as follows:
 	$$
 	p\in T_1\Leftrightarrow
 	\left\{
 	\begin{array}{c}
 	p\in \operatorname{Sing}(S')\\
 	\textrm{ or }\\
 	S' \textrm{ is not transversal to the fiber $\pi_1^{-1}(\pi_1(p))$  at $p$.}
 	\end{array}
 	\right.
 	$$
 Let us decompose $T_1$ as a finite union of closed analytic subsets
 $$
 	T_1= \Sigma_1\cup \Sigma_2\cup\cdots\cup \Sigma_{\ell}\cup K_1,
 $$
 where the $\Sigma_j$ are some of the  irreducible components of $S'$ having codimension $1$ in  $E^1_\infty(\epsilon)$ and $K_1$ has codimension $\geq 2$ in $E^1_\infty(\epsilon)$.
 \begin{lemma}
Any $\pi_1(\Sigma_j)$ is a proper analytic subset of $Y^\epsilon_0$, $j=1,2,\ldots,\ell$.
 \end{lemma}
\begin{proof} Noting that $\Sigma_j$ is a hypersurface of $E^1_\infty$, the non transversality to the fibration  $
	E_\infty^1(\epsilon)\rightarrow Y_0^\epsilon
	$ implies that $\Sigma_j$
	 is invariant under the  foliation associated to this fibration.
	 That is
$
\Sigma_j=\pi_1^{-1}(\pi_1(\Sigma_j))
$.
Then $\pi_1(\Sigma_j)$ is a proper closed analytic subset of $Y_0$, since $\Sigma_j\ne E^1_\infty(\epsilon)$.
\end{proof}
Taking an appropriate intersection of dense open sets,
We conclude that there is a dense open set $V\subset Y^\epsilon_0$ such that for any $p\in V$ the following holds:
\begin{enumerate}
	\item  The codimension of
	$Z_1^\epsilon\cap\pi_1^{-1}(p)$ in $\pi^{-1}(p)$ is greater or equal than two.
	\item The codimension of
	$K_1\cap\pi_1^{-1}(p)$ in $\pi^{-1}(p)$ is greater or equal than two. (Same argument as before)
	\item $\Sigma_j\cap \pi_1^{-1}(p)=\emptyset$, for $j=1,2,\ldots,\ell$.
\end{enumerate}
Let us choose once for all one of such points $p\in V$. For a nonempty Zariski open set of lines $L\subset {\mathbb P}^{n-d-1}_{\mathbb C}=\pi_1^{-1}(p)$ we have that:
\begin{enumerate}
	\item   $L\cap Z_1^\epsilon=\emptyset$.
	\item   $L\cap T_1^\epsilon=\emptyset$.
	\item  For any $q'\in L\cap S'$, we have that $S'\cap \pi_1^{-1}(p)$ is a non singular hypersurface of $\pi_1^{-1}(p)$ that meets $L$ transversally at $q'$.
\end{enumerate}
Take one of those lines $L$ and consider a two dimensional strict transversal $$
(\Delta_0,p)\subset (M_0^\epsilon,p),
$$
such that $\Delta_1\cap \pi_1^{-1}(p)=L$, where $\Delta_1\subset M_1^\epsilon$ is the strict transform of $\Delta_0$ by $\pi_1$.

\begin{lemma}
	The surface  $\Delta_1$  cuts  the adapted singular locus only at points of two-equireduction. Moreover, the intersection is transversal at those points.
\end{lemma}
\begin{proof} The first statement is a direct consequence of the fact that $L$ does not intersect  $Z_1^\epsilon$. Indeed, the intersection of $\Delta_0$ with $\operatorname{Sing}({\mathcal F}^\epsilon_0,E^0(\epsilon))$ is $\{p\}$; hence
	$$
	\Delta_1\cap \operatorname{Sing}({\mathcal F}^\epsilon_1,E^1(\epsilon))=
	 L\cap \operatorname{Sing}({\mathcal F}^\epsilon_1,E^1(\epsilon)).
	$$
	Now,take a point $q'\in L\cap \operatorname{Sing}({\mathcal F}^\epsilon_1,E^1(\epsilon))$. We already know that $q'$ is a two-equireduction point for ${\mathcal F}_1^\epsilon,E^1(\epsilon)$. We have to show the transversality property. Let us note that $ \operatorname{Sing}({\mathcal F}^\epsilon_1,E^1(\epsilon))=S'$ locally at $q'$, and hence we have to show that
	$
	T_{q'}M_1=T_{q'}\Delta_1 + T_{q'} S'
	$.
The transversality between $L$ and $S'\cap\pi_1^{-1}(p)$ at $q'$ in $\pi_1^{-1}(p)$ allows us to choose a vector
$$
v_1\in T_{q'}(L)\setminus T_{q'}(S'\cap\pi_1^{-1}(p)).
$$
Recalling that $q'\notin T_1$, there is transversality between $\pi^{-1}(p)$ and $S'$ at $q'$; therefore, we have
$$
T_{q'}(S'\cap \pi^{-1}(p))=T_{q'}S'\cap T_{q'}(\pi^{-1}(p)).
$$
Since $v_1\in T_{q'}L\subset T_{q'}(\pi^{-1}(p))$, we have $v_1\notin T_{q'}S'$. Let us consider now a vector
$
v_2\in T_{q'}\Delta_1\setminus T_{q'}E^1_\infty(\epsilon)
$.
Noting that $T_{q'}S'$ has dimension $n-2$, we obtain that
$
T_{q'}M_1=<v_1,v_2> + T_{q'} S'= T_{q'}\Delta_1 + T_{q'} S'
$.
\end{proof}
\subsection{End of the proof} We recall that we are assuming that the strict transforms $E_j$, $j\in J_0$, of the components $E^0_j$ of $E^0$ do not intersect $C$. Then we have:
$$
E^1=E^1_\infty\cup \bigcup_{j\in J_0}E^1_j;\quad E= E'\cup \bigcup_{j\in J_0}E_j,
$$
where the $E^1_j\subset E^1$ are the strict transforms of $E^0_j$ in $M_1$ and $E'$ is the union of the strict transforms of the exceptional divisors of the blow-ups (the exceptional part of $\pi$). Note that the strict transform $E_\infty$ of $E^1_\infty$ is contained in $E'$. Let us also recall that we have the hypothesis
that $E_\infty$  belongs to $C$.

Applying Proposition \ref{pro:transveresalsection} to $\Delta_1$, we obtain a reduction of singularities
 $$
 \sigma \vert_{\Delta}: (\Delta, \Delta\cap E, {\mathcal F}\vert_{\Delta})\rightarrow
 (\Delta_1, \Delta_1\cap E^1, {\mathcal F}_1\vert_{\Delta_1}),
 $$
 where $\Delta$ is the strict transform of $\Delta_1$ by $\sigma$. Note that $\Delta$ is also equal to the strict transform of $\Delta_0$ by $\pi$.
 Noting that $\Delta_0$ is a strict transversal to ${\mathcal F}_0$ and that $T_0Y_0\cap T_0\Delta_0=\{0\}$, we obtain by restriction a reduction of singularities
 $
 \pi\vert_\Delta
 $ given by
 $$
 \pi\vert_\Delta=
 \pi_1\vert_{\Delta_1}\circ
 \sigma\vert_\Delta:   (\Delta, \Delta\cap E', {\mathcal F}\vert_{\Delta})\rightarrow
 (\Delta_1, \Delta_1\cap E^1_\infty, {\mathcal F}_1\vert_{\Delta_1})
 \rightarrow
 (\Delta_0, \emptyset, {\mathcal F}_0\vert_{\Delta_0}).
 $$
 Let $C_\Delta$ be the connected component intersecting $E_\infty\cap \Delta$ of the set
 $$
 E'\cap \Delta\setminus \{\text{ nodal corners of } ({\mathcal F}\vert_\Delta, E'\cap \Delta)\}.
 $$
\begin{lemma}
	 $
	C_\Delta\subset C
	$.
\end{lemma}
\begin{proof} It is enough to show that any irreducible component $E'_j\cap \Delta$ of $E'\cap \Delta$ that meets $C_\Delta$ also meets $C$. We know that there is a finite sequence
	$$
	E_\infty= \Delta\cap E'_{j_0} , \Delta\cap E'_{j_1},  \Delta\cap E'_{j_2}, \ldots, \Delta\cap E'_{j_r}=\Delta\cap E'_{j}
	$$
such that $\Delta\cap E'_{j_{m-1}}\cap E'_{j_m}$ is not a nodal point for ${\mathcal F}\vert_\Delta$, for $m=1,2,\ldots,r$. These points do not belong to the separator set of $({\mathcal M},{\mathcal F})$ and hence all the components $E'_{j_m}$ meet $C$, for $m=0,1,2,\ldots,r$. In particular $E'_j$ meets $C$.
\end{proof}
By the two-dimensional result in \cite[Corollary 4.1]{Ort-R-V}, there is an invariant analytic branch $(\Gamma'_0,p)\subset (\Delta_0,p)$ such that the strict transform $(\Gamma',s')\subset \Delta$ meets $C_\Delta$. Moreover, we have that $(\Gamma', s')$ is not contained in the
strict transform of $E^0$, since $s'\in C_\Delta\subset C$ and $E_j\cap C=\emptyset$ for any $j\in J_0$. A ``fortiori'' we have that $(\Gamma',s')$ is not contained in $E$. This implies that $s'$ is a trace singular point in $C$. If is contained in an irreducible component $X$ of type trace of $\operatorname{Sing}({\mathcal F},E)$ that is entirely contained in $C$. Now, we ca take a point $s\in X\cap \pi^{-1}(0)$. The point $s$ being of trace type, we can find a germ of invariant curve $(\Gamma,s)$ not contained in $E$. Now we are done by taking $(\Gamma_0,0)$ to be the image of $(\Gamma,s)$ under $\pi$.
\section{Appendix: Simply Connectedness}
\label{Appendix: Simply Connectedness}
Let us give in this Appendix an outline of a proof of Proposition \ref{prop:simplyconnected}. We consider a combinatorial strata structure  ${\mathcal H}\subset {\mathcal P}(I)$ and we have to prove
\begin{enumerate}
\item[a)] The topological space $\Omega_{\mathcal H}$ is connected if and only if ${\mathcal H}$ is $1$-connected.
 \item[b)]  The combinatorial strata structure $\mathcal H$ is simply connected if and only if $\Omega_{{\mathcal H}_3}$ is a simply connected topological space.
\end{enumerate}
Let us give a proof of a). The simplicial complex $\Omega_{\mathcal H}$ is connected if and only if any two given vertices $\xi_i$ and $\xi_j$ may be connected by a topological path; indeed any point is $\Omega_{\mathcal H}$ is connected with a vertex by a topological path. This already shows that if $\mathcal H$ is $1$-connected, then $\Omega_{\mathcal H}$ is a connected topological space. Conversely, assume that $\mathcal H$ is not $1$-connected, and consider the decomposition in connected components ${\mathcal H}=\cup_{\lambda\in \Lambda} {\mathcal H}_\lambda$ as in Remark \ref{rk:connected components}. We have that
$$
\Omega_{\mathcal H}=\cup_{\lambda\in \Lambda}\Omega_{{\mathcal H}_\lambda};\quad
\Omega_{{\mathcal H}_\lambda}\cap \Omega_{{\mathcal H}_{\lambda'}}=\emptyset, \text{ if }\lambda\ne\lambda'.
$$
Hence $\Omega_{\mathcal H}$ is not connected, since it is a disjoint union of finitely many (at least two) simplicial complexes.

Let us give a proof of b). In view of Remark  \label{rk:tresdimensionalcase}, we assume that ${\mathcal H}={\mathcal H}_3$ and in view of part a) we also assume that $\mathcal H$ and hence $\Omega_{\mathcal H}$ are connected. Let us consider the combinatorial fundamental group $\pi_1({\mathcal H},\{i_0\})$, whose elements are the homotopy classes of loops, that is constructed ``mutatis mutandis'' as the clasical Poincaré group. Now it is enough to prove that
\begin{equation}
\label{eq:piuno}
\pi_1({\mathcal H},\{i_0\})=\pi_1(\Omega_{\mathcal H}, \xi_{i_0}).
\end{equation}
We prove the equality in Equation \ref{eq:piuno} by induction on the lexicographical counter $(\sharp I,\sharp {\mathcal H})$. The starting case $(1,1)$ corresponds to a single point and in this case both groups are trivial. In order to do the induction step, let us separate the cases ${\mathcal H}(3)=\emptyset$ and ${\mathcal H}(3)\ne\emptyset$.

Assume that ${\mathcal H}(3)=\emptyset$ and $\sharp I\geq 2$. Note that in this case $\Omega_{\mathcal H}$ is a connected union of linear segments.  Given $i\in I$, denote $\operatorname{Star}(\mathcal H,i)$ the set
$$
\operatorname{Star}(\mathcal H,i)=\{j\in I\setminus\{i\}; \{i,j\}\in \mathcal H\}.
$$
In there is $i\in I$ such that $\sharp \operatorname{Star}(\mathcal H,i)=1$, we can consider
${\mathcal H}'={\mathcal H}\cap {\mathcal P}(I\setminus\{i\})$. Taking $i_0\ne i$, we see directly that $\pi_1({\mathcal H}',\{i_0\})=\pi_1({\mathcal H},\{i_0\})$ and $\pi_1(\Omega_{{\mathcal H}'},\xi_{i_0})=\pi_1(\Omega_{{\mathcal H}},\xi_{i_0})$; we are done by induction. Assume that $\sharp \operatorname{Star}(\mathcal H,i)\geq 2$ for any $i\in I$. Let us choose $i_1,i_0\in I$, with $i_1\ne i_0$. Consider ${\mathcal H}'$ defined as follows:
$$
{\mathcal H}'=({\mathcal H}\cap {\mathcal P}(I\setminus\{i_1\}))\cup {\mathcal P}(\operatorname{Star}(\mathcal H,i_1))_2.
$$
That is, we eliminate $i_1$ and we add all the two by two connections between the elements of $\operatorname{Star}(\mathcal H,i_1)$. We see in a direct way that $\pi_1({\mathcal H}',\{i_0\})=\pi_1({\mathcal H},\xi_{i_0})$ and, by means of a deformation retract, that $\pi_1(\Omega_{{\mathcal H}'},\{i_0\})=\pi_1(\Omega_{\mathcal H},\xi_{i_0})$; as before, we are done by induction.

Assume that ${\mathcal H}(3)\ne\emptyset$. There is a stratum $J=\{i_1,i_2,i_3\}\in {\mathcal H}(3)$. In this case, we have that ${\mathcal H}= {\mathcal H}'\cup {\mathcal P}(J)$, where ${\mathcal H}'={\mathcal H}\setminus\{J\}$. Hence we have that
$$
\Omega_{\mathcal H}=\Omega_{{\mathcal H}'}\cup \Delta_J,
$$
where $\Omega_{{\mathcal H}'}=\Delta_{J}$ is connected and, more precisely, it is the frontier $\partial\Delta_J$ of $\Delta_J$.
By Classical Seifert Van Kampen theorem, we know that $\pi_1(\Omega_{\mathcal H},\xi_{\{i_1\}})$ is isomorphic to $\pi_1(\Omega_{\mathcal H}',\xi_{\{i_1\}})$ quotient by the normal subgroup generated by a single loop $\sigma$ supported by $\partial\Delta_J$. We use a specific combinatorial version of Seifert-Van Kampen theorem applied to the decomposition ${\mathcal H}= {\mathcal H}'\cup {\mathcal P}(J)$, where ${\mathcal H}'={\mathcal H}\setminus\{J\}$, to show that
$\pi_1({\mathcal H},\{i_1\})$ is the quotient of $\pi_1({\mathcal H}',\{i_1\})$ by the normal subgroup generated by the single loop $\sigma=(\{i_1\},\{i_2\},\{i_3\},\{i_1\})$, In this way, we end by induction.

Let us give an idea about the proof of the combinatorial Seifert-Van Kampen result we need.  In view of the classical proof of Van Kampen theorem, see \cite{Mas}, we have to proof that any group morphism $\phi: \pi_1({\mathcal H}',\{i_1\})\rightarrow H$ such that $\phi([\sigma]')=e_H$ extends in a unique way to a group morphism $F: \pi_1({\mathcal H},\{i_1\})\rightarrow H$. The uniqueness is clear and necessarily defined by
$$
F([\gamma])=\phi([\gamma]').
$$
Hence it is enough to show that if $[\gamma]=[c_{\{i_1\}}]$ then $\phi([\gamma]')=e_H$. Take a finite sequence of elementary homotopies in ${\mathcal H}$
$$
c_{\{i_1\}}=\gamma_0\sim \gamma_1\sim \gamma_2\sim\cdots\sim\gamma_s=\gamma.
$$
If none of them correspond to homotopic pairs $(\epsilon,\epsilon')$ obtained from the set $J={i_1,i_2,i_3}$, the homotopies are also homotopies in ${\mathcal H}'$ and we are done. Assume that
$$
\gamma_j=\alpha_j*\epsilon*\beta_j,\quad  \gamma_{j+1}=\alpha_j*\epsilon'*\beta_j.
$$
Then $\gamma_{j+1}*\gamma_j^{-1}= \alpha_j*(\epsilon'*\epsilon)*\alpha^{-1}$ and then $\phi([\gamma_{j+1}*\gamma_j^{-1}])=e_H$. Hence $\phi([\gamma_{j+1}]')=\phi([\gamma_j]')$ and we are done.

\section{Appendix: Strong Desingularization in Dimension Three}
\label{Appendix: Strong Desingularization in Dimension Three}
Let us give here an outline for a proof of the statement in Remark \ref{rk:reductionoflists} in the case of a three-dimensional ambient space. It is based on the classical methods by Hironaka, Abhyankar and others for reduction of singularities in small dimensions, presented in \cite{Cos-G-O} (See also \cite{Aro-H-V} and \cite{Hir2}).

We start with a pair $({\mathcal M},{\mathcal L})$, where ${\mathcal M}=(M,E;K)$ is a three-dimensional ambient space and ${\mathcal L}$ is a finite list of irreducible hypersurfaces not contained in $E$.
Denote by $H$ the union of the hypersurfaces in ${\mathcal L}$ and take a point $p$ in $K$.
We consider the following local invariants:
\begin{enumerate}
\item The multiplicity $\nu_p(H)$ of $H$ at the point $p$.
\item The dimension $d_p(H)$ of Hironaka's strict tangent space $T_pH$.
\item The number $e_p(E)$ of irreducible components of $E$ through $p$.
\item The ``encombrement'' $t_p(H, E)$ of $T_pH$ with respect to $E$.
\end{enumerate}
The reader is supposed to be familiar with the multiplicity $\nu_p(H)$. Let us note that $\nu_p(H)>0$ if and only if $p\in H$.  The strict tangent space $T_pH$ of $H$ in a point $p\in H$ is the $\mathbb C$-vector subspace of $T_pM$ whose elements are the vectors leaving the tangent cone $C_pH$ invariant by translation. In other words, if $\xi_1,\xi_2,\ldots,\xi_s$ is a basis for the ortogonal $T_pH^\vee\subset T^*_pM$, then $C_pH$ has an equation of the form
$$
\phi(\xi_1,\xi_2,\ldots,\xi_s)=0,
$$
where $\phi$ is an homogeneous polynomial of degree $\nu_p(H)$.

Let us give the definition of $t_p(T_pH, E)$. Choose coordinates $x=(x_1,x_2,\ldots,x_n)$ such that
$E=(x_1x_2\cdots x_e=0)$,  locally at  $p$. (Recall that $n=3$).
Note that each $x_i$ gives a cotangent vector $\bar x_i\in T^*_pM$. Given a subset   $J\subset \{1,2,\ldots,e\}$, we define $t_J$ to be the dimension of the $\mathbb C$-vector subspace $T_J(H,E;p)$ of $T^*_pM$ given by
$$
T_J(H,E;p)=T^\vee_pH\cap <\bar x_i;i\in J> .
$$
For each maximal sequence
$
\sigma: \{1,2,\ldots,e\}=J_1\supsetneq J_2\supsetneq\cdots\supsetneq J_e\supsetneq
\emptyset,
$
we put $t_\sigma=(t_{J_1}, t_{J_2},\ldots,t_{J_e})\in {\mathbb Z}_{\geq 0}^e$. Put $\theta_p(H,E)=t_{J_1}$. Note that
$$
n-d_p(H)\geq
{\theta_p(H,E)}\geq t_{J_2}\geq \cdots\geq t_{J_e}\geq 0.
$$
In particular ${\theta_p(H,E)}=0$ if and only if $t_p(\sigma)=(0,0,\ldots,0)$ for any $\sigma$ (equivalently, for one $\sigma$). Finally, we define $t_p(H,E)$ to be the maximum for the lexicographical ordering of the sequences $t_\sigma$, when $\sigma$ varies.

Let us introduce the number $\zeta_p(H,E)$, obtained from $d_p(H)$, $t_p(H,E)$ and $e_p(E)$. If  If $d_p(H)\in \{0,1\}$, we put $\zeta_p(H,E)=0$. If $d_p(H)=2$, then  $\zeta_p(H,E)$ takes one of the four possible values $0,1,2,3$ as follows:
\begin{enumerate}
\item[-] We say that $\zeta_p(H,E)=0$ if and only if $\theta_p(H,E)=0$.
\item[-] We say that $\zeta_p(H,E)=1$ if and only if $e_p(E)=3$ and $t_p(H,E)=(1,0,0)$.
\item[-] We say that $\zeta_p(H,E)=2$ if and only if $e_p(E)\geq 2$ and $t_p(H,E)=(1,0)$ or $t_p(H,E)=(1,1,0)$.
\item[-] We say that $\zeta_p(H,E)=3$ if and only if $e_p(E)\geq 2$ and $t_p(H,E)=(1,1)$ or $t_p(H,E)=(1,1,1)$.
\end{enumerate}
The main invariant of control $I_p(H,E)$ is the lexicographical invariant
$$
I_p(H,E)=(\nu_p(H),d_p(H), \zeta_p(H,E)).
$$
By convention, we put $I_p(H,E)=(0,0,0)$ when $\nu_p(H)=0$.
The following results of stability are implicitly contained in the above cited works (the reader can prove them by taking a Weierstrass-Tchirnhausen preparation of a local equation of $H$):
\begin{lemma}[Horizontal stability]
\label{lema:estabilidad horizontal}
 The invariant $I_p(H,E)$ is analytically upper semicontinuous.
\end{lemma}
\begin{lemma}[Vertical stability]
 \label{lema:estabilidad vertical}
 Let $\pi: ({\mathcal M}',{\mathcal L}')\rightarrow ({\mathcal M},{\mathcal L})$
be an admissible blowing up with an equimultiple center $Y$ and consider a point $p\in Y$. Let us recall that $\pi^{-1}(p)=\mbox{\rm $\mathbb P$roj }(T_pM/T_pY)$. Then $T_pY\subset T_pH$. Moreover, for any  $p'\in \pi^{-1}(p)$, we have that
\begin{enumerate}
\item $
I_{p'}(H',E')\leq I_p(H,E)
$
for the lexicographical ordering .
\item If $\nu_{p'}(H')=\nu_p(H)$, then $p'\in\mbox{\rm $\mathbb P$roj }(T_pH/T_pY)$.
\end{enumerate}
\end{lemma}

When $({\mathcal M},{\mathcal L})$ is locally simple, we can make it simple just by blowing-up the points where the strata are not connected. Thus our objective is to get a locally simple pair after a suitable admissible transformation.

Let us denote by  $\operatorname{Imax}(H,E)$ the maximum of the invariants $I_p(H,E)$ for $p\in K$. We have that $({\mathcal M},{\mathcal L})$ is locally simple if and only if
$
\operatorname{Imax}(H,E)\leq (1,2,0)
$.
By an elementary induction, our objective is reached if we show how to get  a new $({\mathcal M}',{\mathcal L}')$ such that
$
\operatorname{Imax}(H',E')<
\operatorname{Imax}(H,E)
$, when we start with $
\operatorname{Imax}(H,E)> (1,2,0)
$.
Assume thus that $
\operatorname{Imax}(H,E)=(r, d, \zeta)>(1,2,0)
$ and consider the analytic subset
$$
\operatorname{Sam}_{r,d,\zeta}(H,E)=\{p\in M;\; I_p(H,E)=(r,d,\zeta)\}.
$$
We have to obtain that $\operatorname{Sam}_{r,d,\zeta}(H,E)=\emptyset$ by means of an admissible transformation with equimultiple centers.

$\bullet$ Let us describe first  how to proceed in the cases with $d\leq 1$. Note that in this cases we have $r\geq 2$ and thus the set $Eq_r(H)$ of $r$-multiplicity defined by
$$
Eq_r(H)=\{p\in M; \nu_p(H)\geq r\}=  \{p\in M; \nu_p(H)=r\}
$$
is a finite union of points and curves, that contains $\operatorname{Sam}_{r,d,\zeta}(H,E)$.

Assume we are in the case $d=0$. In this case $(r,d,\zeta)=(r,0,0)$. The sets and  $Eq_r(H,E)$ and $\operatorname{Sam}_{r,0,0}(H,E)$ coincide and they consist in a finite union of points. We blow-up all these points (at the same time or one after the other) and we are done by Lemma \ref{lema:estabilidad vertical}.

Assume we are in the case $d=1$. In this case $(r,d,\zeta)=(r,1,0)$. Any curve $\Gamma$ in $Eq_r(H)$ is also contained in $\operatorname{Sam}_{r,1,0}(H,E)$. In fact, a non isolated point in $Eq_r(H)$ cannot have $0$-dimensional strict tangent space. We can proceed by induction on the number $\alpha$ of irreducible curves contained in  $\operatorname{Sam}_{r,1,0}(H,E)$ which is the same one as the number of $r$-equimultiple curves.

If $\alpha=0$, then $\operatorname{Sam}_{r,1,0}(H,E)$ is a finite set of points. We blow-up one of such points. By Lemma \ref{lema:estabilidad vertical}, at most a new point of multiplicity $r$ may appear, if it does not appear, the number of points in $\operatorname{Sam}_{r,1,0}(H,E)$ decreases by a unit; if it appears in a persistent way, we detect an equimultiple curve,what is impossible.

Assume that $\alpha>0$. By blowing-up points, we obtain that $\operatorname{Sam}_{r,1,0}(H,E)$ has strong normal crossings with $E$, since the curves in $\operatorname{Sam}_{r,1,0}(H',E')$ are the strict transforms of the curves in $\operatorname{Sam}_{r,1,0}(H,E)$. We can assume this property an choose an $r$-equimultiple curve $\Gamma$ as center.   By Lemma \ref{lema:estabilidad vertical}, no $r$-multiple point appears over a point in $\Gamma$. Then $\alpha'=\alpha-1$ and we are done.\\

$\bullet$ Let us consider now the cases with $d=2$. There are four possible situations following the value of $\zeta\in \{0,1,2,3\}$.

$-$ {\em Case $\zeta=0$}: Note that in this case we have $r\geq 2$.
We start by getting strong normal crossings between $\operatorname{Eq}_r(H)$ and $E$ by means of a finite number of blow-ups centered at points in $\operatorname{Eq}_r(H)$. This property is stable under new punctual blowing-ups.

In this situation, if we blow-up a curve $\Gamma\subset\operatorname{Eq}_r(H)$, we do not destroy the property that  $\operatorname{Eq}_r(H)$ and $E$ have strong normal crossings. To see this, we can use the existence of local maximal contact provided by a Weierstrass Tchirnhausen presentation of an equation of $H$. More precisely, given a point $p\in \operatorname{Eq}_r(H)$, there are local coordinates
$(x_1,x_2,z)$  such that $E\subset (x_1x_2=0)$,  an equation of $H$ has the form
$$
z^r+f_2(x_1,x_2)z^{r-2}+f_3(x_1,x_2)z^{r-3}+\cdots +f_r(x_1,x_2)=0
$$
and $\operatorname{Eq}_r(H)\subset (z=x_1=0)\cup (z=x_2=0)$. If we blow-up $z=x_1=0$ and it is contained in $\operatorname{Eq}_r(H)\subset (z=x_1=0)\cup (z=x_2=0)$, then the new $\operatorname{Eq}_r(H')$ is contained in the intersection with the exceptional divisor of the strict transform of the maximal contact surface $z=0$.

The global strategy is as follows: if there is a curve $\Gamma$ contained in $\operatorname{Eq}_r(H)$ (intersecting $\operatorname{Sam}_{r,2,0}(H,E)$, but this is not essential);  then blow-up one of such $\Gamma$. Otherwise, blow-up a point in $\operatorname{Sam}_{r,2,0}(H,E)$. To show that this procedure ends in a finite number of steps, the reader may follow the ideas in Hironaka's Bowdoin College Seminar \cite{Cos-G-O}. Roughly speaking, we reduce the global control to a local one along ``bamboes'' and after this,  we use the properties of the evolution of the characteristic polygon to show the finiteness.

$-$ {\em Case $\zeta=1$}: Note that there are only finitely many points $p$ in $\operatorname{Sam}_{r,2,1}(H,E)$, since $e_p(E)=3$ for each of such points $p$. Consider local coordinates $(x_1,x_2,x_3)$ at $p$ such that $E=(x_1x_2x_3=0)$. The initial part of a local equation of $H$ has the form
 $$
 ( x_1+\alpha x_2+\beta x_3)^r;\quad \alpha\beta\ne 0.
 $$
Then, after the blowing-up of $p$, each point $p'$ in the exceptional divisor with $r=\nu_{p'}(H')$ and $2=d_{p'}(H)$ has $\zeta_{p'}(H',E')=0$. Thus, we end by blowing-up one by one the points in  $\operatorname{Sam}_{r,2,1}(H,E)$.

$-${\em Case $\zeta=2$}:
Recall that any point $p\in\operatorname{Sam}_{r,2,2}(H,E)$ has $e_p(E)\geq 2$. Thus $\operatorname{Sam}_{r,2,2}(H,E)$ is contained in the union of the curves $E_{ij}=E_i\cap E_j$, with $i\ne j$. We follow the following strategy: if there is an $E_{ij}\subset \operatorname{Eq}_r(H)$ with $E_{ij}\cap \operatorname{Sam}_{r,2,2}(H,E)\ne\emptyset$, then blow-up one of such $E_{ij}$. Otherwise, blow-up a point $p\in \operatorname{Sam}_{r,2,2}(H,E)$. Note that the centers have strong normal crossings with $E$, since they are points or curves of the type $E_{ij}$.

Let us see what happens when we blow-up a curve $E_{ij}\subset \operatorname{Eq}_r(H)$ containing a point $p\in \operatorname{Sam}_{r,2,2}(H,E)$. Take local coordinates $(x_1,x_2,y)$ such that
$$
(x_1x_2=0)\subset E\subset (x_1x_2y=0), \quad E_{ij}=(x_1=x_2=0).
$$
A local equation $h$ of $H$ has the form
$$
h=(\alpha x_1+\beta x_2)^r+\tilde h, \quad \alpha\beta\ne0,
$$
where $\tilde h$ has generic order $\geq r$ along $x_1=x_2$. After the blow-up of $E_{ij}$, there are no points $p'$ over $p$ with $I_{p'}(H',E')=(r,2,2)$. Then all that curves disappear after finitely many steps.

We are then in a situation without $r$-equimultiple curves $E_{ij}$ that intersect $\operatorname{Sam}_{r,2,2}(H,E)$. We blow then a point $p\in \operatorname{Sam}_{r,2,2}(H,E)$. In local coordinates, the initial part of an equation of $H$ has the form $(\alpha x_1+\beta x_2)^r$, with $\alpha\beta\ne0$ and $(x_1x_2=0)\subset E$. The only possible point $p'\in \operatorname{Sam}_{r,2,2}(H',E')$ over $p$ corresponds to the interesction with the exceptional divisor of the strict transform of $x_1=x_2=0$ and moreover, no new $r$-equimultiple curves of the type $E'_{i'j'}$ will appear. We found that this points disappear after finitely many steps, since otherwise we find that $x_1=x_2=0$ should be $r$-equimultiple.

$-$ {\em Case $\zeta=3$}:
Each point $p\in \operatorname{Sam}_{r,2,3}(H,E)$ selects an irreducible component $E(p)$ of the divisor $E$, given by the following property: there are local coordinates $(x,y,z)$ such that the initial part of an equation of $H$ has the form $x^r$ and $E(p)=(x=0)$. These $E(p)$ will act as maximal contact surfaces. More precisely, consider an admissible blow-up
$$
\pi:({\mathcal M}',{\mathcal L}')\rightarrow ({\mathcal M},{\mathcal L})
$$
centered in $Y$, with $p\in Y$.  Then $Y\subset E(p)$ and any $p'\in \pi^{-1}(p)\cap \operatorname{Sam}_{r,2,3}(H',E')$ satisfies that $E'(p')$ is the strict transform of $E(p)$. In particular, the number of possible $E(p)$ is not increased. Thus, by finite induction it is enough to eliminate one of them. We select an irreducible component $D$ of $E$ and we consider the set
$$
\operatorname{Sam}^D_{r,2,3}(H,E)=\{p\in \operatorname{Sam}_{r,2,3}(H,E); E(p)=D \}.
$$
We want to make disappear this set after finitely many blow-ups. The first step is to make that $E$ and $D\cap \operatorname{Eq}_r(H)$ do have strong normal crossings by blowing-up points. This property is obtained by a classical two dimensional argument, taking $D$ as a new ambient space. The property is stable under blow-up centered in points or in $r$-equimultiple curves contained in $D$. We take now the strategy of blowing-up one first the curves $\Gamma\subset D\cap \operatorname{Eq}_r(H)$ that intersect $\operatorname{Sam}^D_{r,2,3}(H,E)$ and when there is no one, we chose as center a point in $\operatorname{Sam}^D_{r,2,3}(H,E)$. The classical control by the characteristic polygon assures that a procedure following this strategy stops in finitely many steps.

\end{document}